\documentclass{article}
\usepackage{graphicx} 
\usepackage{longtable}
\usepackage{booktabs} 
\usepackage{amsmath, amssymb, amsfonts, amsthm, mathrsfs}

\usepackage{graphicx}
\usepackage{subcaption}
\usepackage{booktabs}
\usepackage{float}
\usepackage{afterpage}
\usepackage[hidelinks]{hyperref}
\usepackage{bookmark}
\usepackage{xcolor}
\usepackage[capitalize]{cleveref}

\usepackage{enumitem}
\usepackage{geometry} 
\usepackage{natbib}

\newtheorem{theorem}{Theorem}[section]
\newtheorem{lemma}[theorem]{Lemma}

\newtheorem{definition}[theorem]{Definition}

\theoremstyle{definition}  
\newtheorem{example}{Example}[section]
\newtheorem{remark}{Remark}
\theoremstyle{plain}  
\newtheorem*{assumption}{Assumption}

\title{Onsager--Machlup Functional for Fractional Stochastic Newton Dynamics with Time-Dependent Noise Intensities}

\author{
Yanbin Zhu%
\thanks{This work is supported by the National Key R\&D Program of China (No.~2023YFA1009200), the National Key Project of the National Natural Science Foundation of China (Grant No.~12531009), and the National Natural Science Foundation of China (Grant Nos.~12471183, 12071175).}
\and
Xiaomeng Jiang%
\thanks{Corresponding author.}
\and
Yong Li
}

\date{}

\begin{document}
\maketitle

\begin{center}
College of Mathematics, Jilin University, Changchun 130012, PR China
\end{center}

\begin{center}
\texttt{zhuyb23@mails.jlu.edu.cn},
\texttt{jxmlucy@hotmail.com},
\texttt{liyong@jlu.edu.cn}
\end{center}

\begin{abstract}
In this paper, we derive the Onsager--Machlup functional for a second-order Newton-type stochastic system driven by time-dependent fractional noise,
\[
    X_t'' = f_t(X_t, X_t') + \sigma_t \,\xi_t^{H},
\]
where \( H \in (1/4,1) \). The analysis relies on applying a Girsanov transformation to the non-degenerate components and evaluating the limiting conditional expectation associated with the noise term, for which the stochastic Fubini theorem plays a crucial role. To illustrate the applicability of the result, we study two mechanical systems perturbed by noise and provide supporting numerical simulations.
\end{abstract}

\medskip
\noindent\textbf{Keywords.}
Fractional Brownian motion; Onsager--Machlup functional; Stochastic mechanical systems.

\medskip
\noindent\textbf{MSC (2020).}
60H10; 60G22; 60F10.

\allowdisplaybreaks

\section{Introduction}
The second-order Newton equation
\begin{equation*}
    X''_t = f_t(X_t, X'_t)
\end{equation*}
is a cornerstone of classical mechanics, governing the motion of systems ranging from celestial bodies to microscopic particles (see \cite{Arnold1978}). Classical examples include the pendulum equation
\[
X''_t = -\gamma X_t + \sin(X_t),
\]
and the Duffing equation
\[
X''_t + \delta X'_t + \alpha X_t + \beta X_t^3 = \gamma \cos(\omega t),
\]
both of which hold significant importance in both mathematical theory and its applications.

In many physical settings, Newtonian systems are subject to random perturbations arising from environmental fluctuations. Such effects are commonly modeled by incorporating white noise into the governing equations, leading to a stochastic second-order dynamics of the form
\[
X_t'' = f_t\bigl(X_t, X_t'\bigr) +  \xi_t,
\]
where $\xi_t$ represents a fluctuating random force.

The dynamical behavior of such noise-perturbed Newtonian systems was first introduced by Paul Langevin in his seminal 1908 work~\cite{Langevin1908}, and has since been extensively developed in both physics and applied mathematics.

In this work, we investigate the case of time-dependent fractional noise described by
\begin{equation}
    X''_t = f_t(X_t, X'_t) + \sigma_t \, \xi^H_t, \quad (X_0, X'_0) = (x_0, y_0), \label{fir}
\end{equation}
where \(\xi^H_t\) denotes fractional noise with Hurst index \(H\). Compared to the classical setting, this formulation can capture systems subject to noise perturbations with both time dependence and memory effects.

By introducing the auxiliary variable \(Y_t = X'_t\), equation \eqref{fir} can be rewritten equivalently as the following two-dimensional SDE with degenerate noise:
\begin{equation}
\label{sec}
\begin{cases}
dX_t = Y_t \, dt, \\[4pt]
dY_t = f_t(X_t, Y_t) \, dt + \sigma_t \, dB^H_t,
\end{cases}
\end{equation}
subjected to the initial condition \((X_0, Y_0) = (x_0, y_0)\).

For stochastic dynamical systems, a central problem is to characterize their most probable transition paths, which play a crucial role in understanding transition phenomena in fields such as chemical kinetics and non-equilibrium statistical mechanics \cite{VandenEijnden2010}. The Onsager-Machlup (OM) functional offers a powerful framework for this purpose. Given a path \(\psi\), if the limit
\begin{equation}
\label{omf}
    \exp\bigl(J(\psi)\bigr) = \lim_{\varepsilon \to 0} 
    \frac{\mathbb{P}\bigl(\| X_{\cdot} - \psi_{\cdot} \|_{1+\beta} \leq \varepsilon\bigr)}
         {\mathbb{P}\bigl(\| \textstyle\int_0^{\cdot} \sigma_s \, dB^H_s \|_\beta \leq \varepsilon\bigr)}
\end{equation}
exists, where \(\|\cdot\|_{\beta}\) denotes a suitable Hölder norm (see Section~2 for precise definitions), then \(J\) is called the OM functional associated with equation \eqref{fir}. Intuitively, this functional describes the probability that the system solution \(X_t\) lies within a tubular neighborhood of a deterministic path \(\psi_t\) as \(\varepsilon\) approaches zero. Its extremal paths correspond to the most probable transition pathways of the stochastic dynamics.

The OM functional was first introduced by Onsager and Machlup \cite{Onsager1953} in 1953. It characterizes the most probable paths of diffusion processes and serves as an analogue to the Lagrangian in dynamical systems, describing the optimal trajectory of a particle along a given path. This theory has  been widely applied to compute the most probable reaction paths from reactant to product states, see \cite{BATTEZZATI2013163,10.1063/1.467139}.

For SDEs driven by fractional noise, the Onsager--Machlup functional has been studied by Liu and Gao\cite{Liu2026}, 
and related results for time-varying fractional noise have been investigated by the authors recently.  However, the system \eqref{sec} presents new challenges due to the degeneracy of the noise and the indirect influence of the noise on certain stochastic integral terms. Specifically, the standard Girsanov transformation cannot be applied directly, and the additional term arising from Taylor expansions requires careful treatment.

To address these challenges, we apply the Girsanov transformation solely to the non-degenerate component \(Y_t\) and characterize the distance between \(X_t\) and \(\psi_t\) in terms of their derivatives. For the additional term, we employ the stochastic Fubini theorem together with careful estimates of trace-type terms.

Our main results (Theorems~\ref{result1}, \ref{result2}, and \ref{result3}) provide explicit formulas for the OM functional in three distinct regimes:
\[
J(\psi) =
\begin{cases}
-\dfrac{1}{2}\displaystyle\int_0^1 
\Big(
\dot{\phi}_s
- s^{\alpha} D_{0^+}^\alpha s^{-\alpha} 
\sigma_s^{-1} f_s(\psi_s,\phi_s)
\Big)^2 
+ d_H \partial_y f_s(\psi_s,\phi_s) \, ds, 
& \tfrac{1}{2}<H<1, \\[2.2em]
-\dfrac{1}{2}\displaystyle\int_0^1 \frac{
\big(
\phi_s'
- f_s(\psi_s,\phi_s)
\big)^2}{\sigma_s^2} 
+ \partial_y f_s(\psi_s,\phi_s) \, ds,
& H=\tfrac{1}{2}, \\[2.2em]
-\dfrac{1}{2}\displaystyle\int_0^1 
\Big(
\dot{\phi}_s
- s^{-\alpha} I_{0^+}^\alpha s^{\alpha} 
\sigma_s^{-1} f_s(\psi_s,\phi_s)
\Big)^2 
+ d_H \partial_y f_s(\psi_s,\phi_s) \, ds,
& \tfrac{1}{4}<H<\tfrac{1}{2},
\end{cases}
\]
where \(\alpha = H - 1/2\), \(d_H\) is a constant depending on \(H\), and the velocity of a path \(\psi\) is denoted by \(\phi_t = \psi_t'\). Moreover, the functions $\phi$ and $\dot{\phi}$ satisfy $ 
    \phi(\cdot) - y_0 = K_H^\sigma(\dot{\phi})(\cdot),$ 
where $K_H^\sigma$ is an integral operator of order $H+\tfrac{1}{2}$, which will be defined later. Moreover, in the case $H=\tfrac{1}{2}$, we have $
    \dot{\phi}(\cdot) = \frac{\phi(\cdot)}{\sigma(\cdot)}.$

These formulas extend the classical Onsager–Machlup theory to fractional stochastic Newtonian dynamics with time-dependent noise intensities, thereby establishing a foundation for computing the most probable transition paths in applications.

The remainder of this paper is organized as follows. Section 2 introduces the notations and preliminary results. In Section 3, we present the main results concerning the OM functional, together with its application to  stochastic pendulum equation,  stochastic Duffing equation, and accompanying numerical simulations. Finally, Section 4 provides the proofs of the main results.

\section{Preliminaries}
In this section, we recall some basic notations, assumptions, and lemmas that will be used in the sequel.
\subsection{Function Spaces and Norms}
    Let $0<\beta\le 1$. We denote by $C^\beta([0,1])$ the space of continuous functions
$f:[0,1]\to\mathbb{R}$ such that
\begin{equation}\label{eq:Holder-seminorm}
    [f]_\beta := \sup_{0\le s<t\le 1} \frac{|f(t)-f(s)|}{|t-s|^\beta} < \infty.
\end{equation}
We equip $C^\beta([0,1])$ with the norm
\begin{equation*}
    \|f\|_{C^\beta([0,1])} := \sup_{t\in[0,1]} |f(t)| + [f]_\beta.
\end{equation*}

Furthermore, we define
\begin{equation*}
    C^\beta_0([0,1]) := \{ f\in C^\beta([0,1]) : f(0)=0 \}.
\end{equation*}
On this subspace we use the norm
\begin{equation*}
    \|f\|_\beta := [f]_\beta.
\end{equation*}
It is straightforward to check that $\|\cdot\|_\beta$ is indeed a norm on
$C^\beta_0([0,1])$, and that it is equivalent to the restriction of 
$\|\cdot\|_{C^\beta([0,1])}$ to $C^\beta_0([0,1])$.
Further, for \(1 < \beta < 2\), we can define the norm of \(C_0^{\beta}([0,1])\) as  
\[
\|f\|_{\beta} := [f']_{\beta-1}.
\]

\subsection{Fractional Calculus}

We first introduce the basic concepts of fractional calculus. For further details, we refer to \cite{Samko1993}.
 \begin{definition}
   	Let $f\in L^1([a,b])$. The integrals 
   	\begin{align*}
   		(I_{a^+}^\alpha f)(x)&:=\frac{1}{\Gamma(\alpha)}\int_a^x (x-y)^{\alpha-1}f(y)dy,\quad x\geq a,\\
   		(I_{b^-}^\alpha f)(x)&:=\frac{1}{\Gamma(\alpha)}\int_x^b (y-x)^{\alpha-1}f(y)dy,\quad x\leq b,
   	\end{align*}
   	where $\alpha>0$, are respectively called the right and left Riemann--Liouville fractional integrals of order $\alpha$.

   \end{definition}
   
   For any $\alpha\geq 0,$ any $f\in L^p([a,b])$ and $g\in L^q([a,b])$ where $1/p+1/q\leq \alpha,$ we have:
   \begin{equation}
   	\int_a^b f(s)(I_{a^+}^\alpha g)(s)ds=\int _a^b (I_{b^-}^\alpha f)(s) g(s) ds.\label{ffubini}
   \end{equation}
   
   	
   	If $1 \leq p < \infty$, we denote by $I_{a^+}^\alpha(L^p)$ the image of $L^p([a,b])$ under the operator $I_{a^+}^\alpha$. Similarly, $I_{b^-}^\alpha(L^p)$ can be defined.

   	\begin{definition}
   		 Let $f\in I_{a^+}^\alpha(L^p) $, $g\in I_{b^-}^\alpha(L^p) $.
   		 Each of the expressions 
   		 \begin{align*}
   		(D_{a^+}^\alpha f)(x)&:=\left(\frac{d}{dx} \right)^{[\alpha]+1}I_{a^+}^{1+[\alpha]-\alpha}f(x),\\
   		(D_{b^-}^\alpha g)(x)&:=\left(-\frac{d}{dx} \right)^{[\alpha]+1}I_{b^-}^{1+[\alpha]-\alpha}g(x)
   	   	\end{align*}
   		 are respectively called the right and left fractional derivative.
   \end{definition}
  From \eqref{ffubini}, we deduce the formula  
\begin{equation}
  \int_a^b f(s)(D_{a^+}^\alpha g)(s)ds
  =\int _a^b (D_{b^-}^\alpha f)(s) g(s) ds,\quad 0<\alpha<1, \label{fffubini}
\end{equation}
which holds under the assumptions that \( f\in I^\alpha_{b^-}(L^p) \) and \( g\in I^\alpha_{a^+}(L^q) \) satisfying \( 1/p+1/q\leq 1+\alpha \).
   
	If $f\in I_{a^+}^\alpha(L^p) $, the function $ \phi $ such that $f=I_{a^+}^\alpha(\phi)$ is unique in $L^p $. Fractional derivatives can be regarded as the inverse operation of fractional integrals.
 
 When $\alpha p > 1$, any function in $I_{a^+}^\alpha(L^p)$ is 
$(\alpha - \tfrac{1}{p})$-H\"older continuous. Moreover, every H\"older continuous 
function of order $\beta > \alpha$ admits a fractional derivative of order 
$\alpha$; see \cite[Proposition~2.1]{Decreusefond1999}.

Although fractional derivatives are originally defined as derivatives of 
fractional integrals, they also admit an explicit representation in terms of 
Weyl’s formula (see \cite[Remark~5.3]{Samko1993}):
\begin{equation}
D^\alpha_{a^+} f(x) 
= \frac{1}{\Gamma(1-\alpha)} \left( \frac{f(x)}{(x-a)^\alpha} 
+ \alpha \int_a^x \frac{f(x) - f(y)}{(x-y)^{\alpha+1}}\,dy \right), 
\quad 0<\alpha<1,
\label{Weyl}
\end{equation}
where the improper integral converges in the $L^p$ sense. Therefore, if \( f \) has H\"older continuity with an exponent strictly greater than \( \alpha \),  the fractional derivative of order \( \alpha \) exists.

  	For the case of the right fractional derivative, we also have a similar representation:
\begin{equation}
D_{T^-}^\alpha f(s) 
= \frac{1}{\Gamma(1-\alpha)} \left(
    \frac{f(s)}{(T-s)^\alpha}
    - \alpha \int_s^T \frac{f(u)-f(s)}{(u-s)^{\alpha+1}}\,du
  \right).
\end{equation}

\subsection{Fractional Stochastic Integration with Respect to Fractional Brownian Motion}
      Let $W=\{W_t,t\in [0,1] \}$ be a Wiener process defined in the canonical probability space $(\Omega,\mathcal{F},\mathbb{P})$, where 
 $\Omega=C_0([0,1])$
   and $\mathbb{P}$ is the Wiener measure.
	A real-valued continuous process $\{B^H_t,\ t\in [0,T]\}$ is called a fractional Brownian motion with Hurst parameter $H\in (0,1)$ if it is a centered Gaussian process with covariance function  
\[
  \mathbb{E}[B^H_tB^H_s]
  = \tfrac{1}{2}\big(|t|^{2H}+|s|^{2H}-|t-s|^{2H}\big)
  := R_H(t,s).
\]
For simplicity, we assume through all that $\alpha = |H - 1/2|$.

When $H=\tfrac{1}{2}$, fractional Brownian motion reduces to the classical Brownian motion.  
However, if $H\neq \tfrac{1}{2}$, it is neither a semi-martingale nor a Markov process. Moreover, its sample paths belong to $C^{H-\varepsilon}_0([0,T])$ $P$-a.s. for every $\varepsilon>0$.

It is well known that an fractional Brownian motion $B^H_t$ admits a Wiener integral representation: there exists a deterministic kernel $K_H(t,s)$ and a standard Brownian motion $W$ such that  
\[
  B^H_t = \int_0^t K_H(t,s)\,dW_s,
\]
where
\[
  K_H(t,s)=
  \begin{cases}
    c_H\, s^{-\alpha}\displaystyle\int_s^t (u-s)^{\alpha-1}u^\alpha\,du, & H>\tfrac{1}{2}, \\[2ex]
    b_H\Big[(t/s)^{-\alpha}(t-s)^{-\alpha}
    +\alpha s^\alpha \int_s^t (u-s)^{-\alpha}u^{-(\alpha+1)}\,du\Big], & H<\tfrac{1}{2},
  \end{cases}
\]
with constants
\[
  c_H=\sqrt{\frac{H(2H-1)}{\mathrm{B}(2-2H,H-\tfrac{1}{2})}}, 
  \qquad 
  b_H=\sqrt{\frac{2H}{(1-2H)\mathrm{B}(1-2H,H+\tfrac{1}{2})}}.
\]

The operator \( K_H \), mapping from \( L^2([0,1]) \) to \( I_{0^+}^{H+\frac{1}{2}}(L^2([0,1])) \), is given by:  
\begin{equation}
	(K_Hh)(t) = \int_0^t K_H(t,s) h(s) \, ds.\label{kh operator}
\end{equation}

 From \cite[Lemma 10]{Nualart}, the operator \( K_H \) can be expressed using fractional integrals as follows:   
\[
(K_H h)(s) = \begin{cases}
\displaystyle
I^{1-2\alpha}_{0^+} \, s^{\alpha} I^{\alpha}_{0^+} \, s^{-\alpha} h, & H <1/2, \\
\displaystyle
I^{1}_{0^+} \, s^{\alpha} I^{\alpha}_{0^+} \, s^{-\alpha} h, & H >1/2.
\end{cases}
\]  
 The inverse operator \((K_H)^{-1}\) is then defined as:  
\begin{align}  
\notag (K_H)^{-1} h &= s^{\alpha} D^{\alpha}_{0^+} \, s^{-\alpha} D^{1-2\alpha}_{0^+} h, \quad H < 1/2, \\  
(K_H)^{-1} h &= s^{\alpha} D^{\alpha}_{0^+} \, s^{-\alpha} h', \quad H > 1/2,\label{g1/2}   
\end{align}    
for all \(h \in I^{H+\frac{1}{2}}_{0^+}(L^2)\). If \(h\) is differentiable, the operator simplifies to:  
\begin{equation} \label{l1/2}  
(K_H)^{-1} h = s^{-\alpha} I^{\alpha}_{0^+} \, s^{\alpha} h', \quad H <1/2.  
\end{equation}  
Moreover, the operator \((K_H)^{-1}\) preserves the adaptability property.

Stochastic integrals of deterministic functions with respect to a Gaussian process are called Wiener integrals.  Stochastic integrals with respect to fractional Brownian motion can be defined by  its Gaussianity. For the step function $I_{[0,t]}(\cdot)$, we assign the inner product 
\[
\langle I_{[0,t]}(\cdot), I_{[0,s]}(\cdot) \rangle_H = R_H(t,s).
\]
We denote by $\mathcal{H}$ the space obtained by completing the step functions under the above inner product.
Then the mapping $\mathcal{I}:I_{[0,t]}\longmapsto B^H_t$ corresponds to an isometry between the step function under the inner product $\langle  \cdot, \cdot \rangle_H $ and the Gaussian variable in $L^2(\mathbb{P})$. By extending this isometry, we obtain our integral $\mathcal{I}$ as an isometry from $\mathcal{H}$ to the space of Gaussian random variables in $L^2(\mathbb{P})$.

In order to illustrate the elements of $\mathcal{H}$ and the relation between the fractional Brownian  integral and the Brownian motion integral, we first introduce the following operator:
\[
  (K_H^*f)(s)=
  \begin{cases}
    c_H\Gamma(\alpha)\, s^{-\alpha}I_{1-}^\alpha\big(s^\alpha f_s\big), & H>\tfrac{1}{2}, \\[1ex]
    b_H\Gamma(1-\alpha)\, s^\alpha D^\alpha_{1^-}\big(s^{-\alpha}f_s\big), & H<\tfrac{1}{2}.
  \end{cases}
\]

\begin{theorem}[\cite{Biagini2008}]
Let $H\in(0,1)$. 
If $\psi\in\mathcal{H}$, we have
\begin{equation*}
  \int_0^T \psi_s\,dB^H_s 
  := \int_0^T (K_H^*\psi)(s)\,dW_s.
\end{equation*}
\end{theorem}
In particular, for $\psi_s=\mathbf{1}_{[0,t]}(s)$, we recover
\begin{equation*}
  B^H_t=\int_0^t K_H(t,s)\,dW_s.
\end{equation*}

Analogous to the It\^o isometry, we have the following isometry for stochastic integrals with respect to fractional Brownian motion.
\begin{lemma}[\cite{Biagini2008}]\label{isometry}
If $f,g\in\mathcal{H}$, then
\begin{equation*}
  \mathbb{E}\!\left(\int_0^T f_s\,\mathrm{d}B^H_s \int_0^T g_s\,\mathrm{d}B^H_s\right)
  =\int_0^T (K_H^*f)(s)\,(K_H^*g)(s)\,\mathrm{d}s.
\end{equation*}
For $H > \tfrac{1}{2}$, we have a more intuitive equality:
\begin{equation}
    \mathbb{E}\!\left(\int_0^T f_s\,\mathrm{d}B^H_s \int_0^T g_s\,\mathrm{d}B^H_s\right)
    =H(2H-1)\int_0^T\int_0^T f_t g_s |s-t|^{2H-2}\,\mathrm{d}s\,\mathrm{d}t.
\end{equation}
\end{lemma}

	In this paper, we assume that $\sigma$ is H\"older continuous of order $\gamma$ with $\gamma+H>1$. Under this condition, the stochastic integral
\[
\int_0^1 \sigma_u \, dB^H_u
\]
is well defined pathwise in the sense of Young integration(see Lemma~\ref{young}), and this definition agrees with the Wiener integral.
 
	\begin{lemma}[\cite{Young1936}]   \label{young}
		For $f\in C^\beta([0,1]), g\in C^\gamma([0,1])$, if $\beta+\gamma>1,$ 
		\begin{equation*}
			\int_0^t \sigma_s dg_s,\quad \int_0^t g_sd\sigma_s 
		\end{equation*}
		are well-defined. Furthermore, we have
		\begin{equation*}
			\int_0^t \sigma_s dg_s=\sigma_t g_t-\sigma_0 g_0-\int_0^t g_sd\sigma_s.
		\end{equation*}
	\end{lemma}
	
	\begin{lemma}[\cite{Young1936}] \label{young2}
Let $f,g:[0,T]\to \mathbb{R}$ be functions such that $f \in C^{\beta}$ 
and $g \in C^{\gamma}$ with $\beta,\gamma \in (0,1)$ and $\beta+\gamma>1$. 
Then  for any $0\leq s<t\leq T$, one has
\begin{equation}
	\left|\int_s^t f_r \, dg_r - f_s (g_t-g_s)\right|
   \leq C_{\alpha,\beta}\, [f]_{\beta}\,[g]_{\gamma}\, |t-s|^{\beta+\gamma}.\label{youngint}
\end{equation}
\end{lemma}

   Below, we introduce the operator \(K_H^\sigma\) via Young integration and the operator \(K_H\). 
Define the operator $K_H^\sigma$ on $L^2([0,1])$, associated with the process $\sigma_\cdot$ and the operator $K_H$, by
\begin{equation}
    (K_H^\sigma f)(t) = \int_0^t \sigma_s \, d(K_H f)(s), \qquad t \in [0,1]. \label{khsigma}
\end{equation}
According to the properties of the Young integral and the $K_H$ operator, it can be concluded that  
the operator $K_H^\sigma$ is injective.
Consequently, we can define its inverse
\[
    (K_H^\sigma)^{-1} : K_H^\sigma\bigl(L^2([0,1])\bigr) \to L^2([0,1]), 
    \qquad 
    f \mapsto (K_H)^{-1}\!\left(\int_0^\cdot \sigma_s^{-1} \, df_s\right).
\]
For any $\phi$ such that $\phi - y_0 \in K_H^\sigma(L^2([0,1]))$, we denote by $\dot{\phi}$ the element of $L^2([0,1])$ satisfying
\begin{equation}
    \phi - y_0 = K_H^\sigma(\dot{\phi}). \label{khs}
\end{equation}

It is worth noting that the operator $K_H$ is an isomorphism from $L^2([0,1])$ onto 
$I_{0^+}^{H+\frac{1}{2}}(L^2([0,1]))$ 
(see \cite[Theorem~2.1]{Decreusefond1999}).  
From the embedding theorem, $K_H(\dot{\phi})$ is $H$-H\"older continuous, and therefore $\phi=y_0+K_H^\sigma(\dot{\phi})$ also enjoys $H$-H\"older continuity by Lemma~\ref{young2}.

 To ensure the validity of our main result, the coefficients in \eqref{fir} must satisfy the following assumptions.  
 \begin{assumption}[A]\label{ass:A}
We impose the following assumptions on the coefficients of \eqref{fir}.

\textbf{(1) Case $1/4<H\leq 1/2:$}
The coefficients \(f\) and \(\sigma\) satisfy:
\begin{enumerate}
    \item $\sigma \in C^1([0,1])$ and $\inf_{0 \leq s \leq 1} |\sigma_s| > 0$. 
    Without loss of generality, we may assume that $0 < m \leq \sigma \leq M$;
    \item $f$ is continuous in $(t,x,y)$, twice continuously differentiable with respect to $x$ and $y$, and bounded.
\end{enumerate}

\textbf{(2) Case $1/2<H<1:$}
    \begin{enumerate}
        \item All the assumptions in the above case are satisfied;
         \item $f$ satisfies the Lipschitz condition
        \[
           |f_t(x_1,y_1) - f_s(x_2,y_2)| \leq L \big( |t-s| + |x_1-x_2|+|y_1-y_2| \big),\quad t,s\in[0,1];
        \]
        \item The following inequality holds:
        \begin{equation}\label{assumpT}
              1 < \frac{m^2 (2\beta+1)\Gamma(1+\beta)^2}
                    {2M^2 \alpha^2 L^2 \Gamma(\beta-\alpha)^2}.
        \end{equation}
        
    \end{enumerate}

\end{assumption}

\begin{remark}
The boundedness of $f$ is required for the application of Girsanov’s theorem. 
In the computation of the OM functional, $f$ and its derivatives 
will automatically remain bounded under the condition 
\[
\Big\| \int_0^\cdot \sigma_s \, dB^H_s \Big\| \leq \varepsilon.
\]
\end{remark}

\begin{remark}
For the assumption \eqref{assumpT} in Case $1/2 < H < 1$, on a more general time interval $[0,T]$, the condition can be replaced by
\[
T^{2H+1}
<
\frac{m^2 (2\beta+1)\Gamma(1+\beta)^2}
     {2M^2 \alpha^2 L^2 \Gamma(\beta-\alpha)^2}.
\]

\end{remark}

\section{Main Results}
In this section, we first present the main results of our study, and verify our results through numerical simulations.
 The details of the computations and derivations leading to the main results are provided in the subsequent section.
Our OM functional takes the following form:
\[
J(\psi)=
\begin{cases}
-\dfrac{1}{2}\displaystyle\int_0^1 
\Big(
\dot{\phi}_s
- s^{\alpha} D_{0^+}^\alpha \big(s^{-\alpha} 
\sigma_s^{-1} f_s(\psi_s,\phi_s)\big)
\Big)^2 
+ d_H \,\partial_y f_s(\psi_s,\phi_s) \, ds, 
& \tfrac{1}{2}<H<1, \\[2.2em]
-\dfrac{1}{2}\displaystyle\int_0^1 
\dfrac{
\big(
\phi_s'
-  f_s(\psi_s,\phi_s)
\big)^2}{\sigma_s^2} 
+  \partial_y f_s(\psi_s,\phi_s) \, ds,
& H=\tfrac{1}{2}, \\[2.2em]
-\dfrac{1}{2}\displaystyle\int_0^1 
\Big(
\dot{\phi}_s
- s^{-\alpha} I_{0^+}^\alpha \big(s^{\alpha} 
\sigma_s^{-1} f_s(\psi_s,\phi_s)\big)
\Big)^2 
+ d_H \,\partial_y f_s(\psi_s,\phi_s) \, ds,
& \tfrac{1}{4}<H<\tfrac{1}{2},
\end{cases}
\]
where \(\alpha = H - 1/2\), \(d_H\) is a constant depending on \(H\), and the velocity of a path \(\psi\) is denoted by \(\phi_t = \psi_t'\). Moreover, $\phi$ and $\dot{\phi}$ are related through
\begin{equation*}
    \phi(\cdot) - y_0 = K_H^\sigma(\dot{\phi})(\cdot),
\end{equation*}
where $K_H^\sigma$ is defined by \eqref{khsigma}. In particular, when $H=\tfrac{1}{2}$, this relation reduces to
\begin{equation*}
    \dot{\phi}(\cdot) = (K_H^\sigma)^{-1}\bigl(\phi(\cdot) - y_0\bigr)
    = \frac{\phi(\cdot)}{\sigma(\cdot)}.
\end{equation*}

Moreover, these three cases can be written in a unified way as
\begin{equation}\label{mr}
    J(\psi)
    = -\frac{1}{2}\int_0^1  
    \Big[
        (K^H_\sigma)^{-1} 
        \Big(
            \phi_u - y_0 - \int_0^u f_v(\psi_v,\phi_v)\,dv
        \Big)
    \Big]^2 (s)
    + d_H \,\partial_y f_s(\psi_s,\phi_s)\,ds.
\end{equation}

If we consider the velocity transitioning from \( y_0 \) at time \( t = 0 \) to \( y_1 \) at time \( t = 1 \), for the case \( \frac{1}{4} < H < \frac{1}{2} \), using variational methods, we obtain the following:
\begin{align*}
& \lim_{\varepsilon\to 0}\frac{J(\psi+\varepsilon\eta)-J(\psi)}{\varepsilon}\\
&= C \int_0^1 2\left(\dot{(\psi')}_s- s^{-\alpha} I_{0^+}^\alpha s^{\alpha} \sigma_s^{-1} f_s(\psi_s,\psi_s') \right)\\
&\qquad\qquad\cdot\left( \dot{(\eta')}_s-s^{-\alpha}I_{0^+}^\alpha s^\alpha \sigma_s^{-1} (\partial_x f_s(\psi_s,\psi_s')\eta_s+ \partial_y f_s(\psi_s,\psi_s')\eta_s') \right)\\ &\qquad\qquad+d_H\left(\partial_{xy}f_s(\psi_s,\psi_s')\eta_s+ \partial_{yy}f_s(\psi_s,\psi_s')\eta_s' \right)ds\\
    &=C \int_0^1 2\left(\dot{(\psi')}_s- s^{-\alpha} I_{0^+}^\alpha s^{\alpha} \sigma_s^{-1} f_s(\psi_s,\psi_s') \right)\\ 
    &\qquad\qquad\cdot\left( s^{-\alpha}I_{0^+}^\alpha s^\alpha \sigma_s^{-1}\eta_s'' -s^{-\alpha}I_{0^+}^\alpha s^\alpha \sigma_s^{-1} (\partial_x f_s(\psi_s,\psi_s')\eta_s+ \partial_y f_s(\psi_s,\psi_s')\eta_s') \right)\\ &\qquad\qquad+d_H\left(\partial_{xy}f_s(\psi_s,\psi_s')\eta_s+ \partial_{yy}f_s(\psi_s,\psi_s')\eta_s' \right)ds\\
    &=C \int_0^1 [2 ((d/ds)^2\sigma_s^{-1} s^\alpha I^\alpha_{1^-}s^{-\alpha}- \partial_x f_s(\psi_s,\psi_s') \sigma_s^{-1} s^\alpha I^\alpha_{1^-}s^{-\alpha}\\
    &\qquad\qquad+ (d/ds)\partial_y f_s(\psi_s,\psi_s')\sigma_s^{-1} s^\alpha I^\alpha_{1^-}s^{-\alpha} )+\left(\dot{(\psi')}_s- s^{-\alpha} I_{0^+}^\alpha s^{\alpha} \sigma_s^{-1} f_s(\psi_s,\psi_s')\right)\\&\qquad\qquad+d_H(\partial_{xy}f_s(\psi_s,\psi_s')-(d/ds)\partial_{yy}f_s(\psi_s,\psi_s')) ]\eta_s ds . \end{align*}
Therefore, the most probable path for the velocity \( \phi_s \) should satisfy the equation:
\begin{align*}
   &\left( 2 \left( \frac{d}{ds} \right)^2 \sigma_s^{-1} s^\alpha I_{1^-}^\alpha s^{-\alpha} - \partial_x f_s(\psi_s, \phi_s) \sigma_s^{-1} s^\alpha I_{1^-}^\alpha s^{-\alpha} + \frac{d}{ds} \partial_y f_s(\psi_s, \phi_s) \sigma_s^{-1} s^\alpha I_{1^-}^\alpha s^{-\alpha}\right)\\
   \cdot &\left( \dot{(\phi_s)}_s - s^{-\alpha} I_{0^+}^\alpha s^{\alpha} \sigma_s^{-1} f_s(\psi_s, \phi_s) \right) + d_H \left( \partial_{xy} f_s(\psi_s, \phi_s) - \frac{d}{ds} \partial_{yy} f_s(\psi_s, \phi_s) \right) = 0
\end{align*}
with the boundary conditions $ (X_0,Y_0) = (x_0,y_0 )$ and $(X_1, Y_1) = (x_1,y_1 )$.
The same equations can be obtained for the regular  case ($1/2<H<1$)
\begin{align*}
   &\left( 2 \left( \frac{d}{ds} \right)^2 \sigma_s^{-1} s^{-\alpha} D_{1^-}^\alpha s^{\alpha} - \partial_x f_s(\psi_s, \phi_s) \sigma_s^{-1} s^{-\alpha} D_{1^-}^\alpha s^{\alpha} + \frac{d}{ds} \partial_y f_s(\psi_s, \phi_s) \sigma_s^{-1} s^{\alpha} D_{1^-}^\alpha s^{\alpha}\right)\\
   \cdot &\left( \dot{(\phi_s)}_s - s^{\alpha} D_{0^+}^\alpha s^{-\alpha} \sigma_s^{-1} f_s(\psi_s, \phi_s) \right) + d_H \left( \partial_{xy} f_s(\psi_s, \phi_s) - \frac{d}{ds} \partial_{yy} f_s(\psi_s, \phi_s) \right) = 0,
\end{align*}
and  the standard case ($H=1/2$)
\begin{align*}
    &((d/ds)^2-\partial_x f_s(\psi_s,\phi_s)+\frac{d}{ds}\partial_x f_s(\psi_s,\phi_s))(2\frac{\phi_s'-f_s(\psi_s,\phi_s)}{\sigma_s^2})\\ +&
    (\partial_{xy}f_s(\psi_s,\phi_s)+\partial_{yy}f_s(\psi_s,\phi_s))=0.
\end{align*}

    \begin{remark}\label{remark}
      It follows from \eqref{mr} that if $\partial_y f(x,y)=C$, then, when the path is constrained to satisfy the boundary conditions $(X(0),X'(0))=(x_0,y_0)$ and $(X(1),X'(1))=(x_1,y_1)$, where $(x_0,y_0)$ and $(x_1,y_1)$ are connected by a solution of the corresponding noiseless system, the resulting most probable path coincides with the noiseless solution trajectory.

    \end{remark}
    
   Next, we illustrate our main results via the following examples. All numerical simulations in this paper were performed using MATLAB. 

\begin{example}[Pendulum equation]
Consider the second-order stochastic differential equation
\begin{equation}\label{example}
    X_t'' = -\gamma X_t' - k \sin(X_t) + \big(\sigma_0 + A\cos(\omega t)\big)\,\xi_t^H,
\end{equation}
where
\[
    k = \frac{1}{2}\left( \frac{\sqrt{\pi}\,\Gamma(1/4)}{\Gamma(3/4)} \right)^{\!2}.
\]
Let $\bar{X}_t$ denote the solution of the corresponding deterministic system
(obtained by setting $\sigma_0 + A\cos(\omega t) \equiv 0$ in \eqref{example}) with initial condition
\[
    \big(\bar{X}_0,\bar{X}'_0\big) = \big(-\tfrac{\pi}{2},0\big).
\]
 For simplicity, we only consider the undamped case with $\gamma = 0$  
 here.
 It is straightforward to verify that
\[
    \big(\bar{X}_1,\bar{X}'_1\big) = \big(\tfrac{\pi}{2},0\big).
\]

We are interested in the most probable path of the SDE which starts from
$(-\pi/2,0)$ at time $t=0$ and reaches $(\pi/2,0)$ at time $t=1$.
Since $\partial_y \sin(x) = 0$, it follows from Remark\ \ref{remark} that
$\bar{X}_t$ is precisely the most probable path under the above boundary conditions. Numerical simulations are shown in Figs. \ref{fig1a}, \ref{fig1b}, and \ref{fig1c}. It is clear that when $H\le \frac12$, our assumptions are satisfied. Moreover, for the choice $H=0.51$, $\beta=0.28$, $\sigma_0=2$, and $A=0.1$, the conditions are also fulfilled. 

    From the figure, we observe that the average path is very close to the most probable path. Furthermore, it is evident in this example that while the noise frequency demonstrates a clear correlation with the SDE's sample paths, it leaves the most probable path unaffected.

\end{example}

\begin{figure}[htbp]
    \centering
    \includegraphics[width=0.9\textwidth]{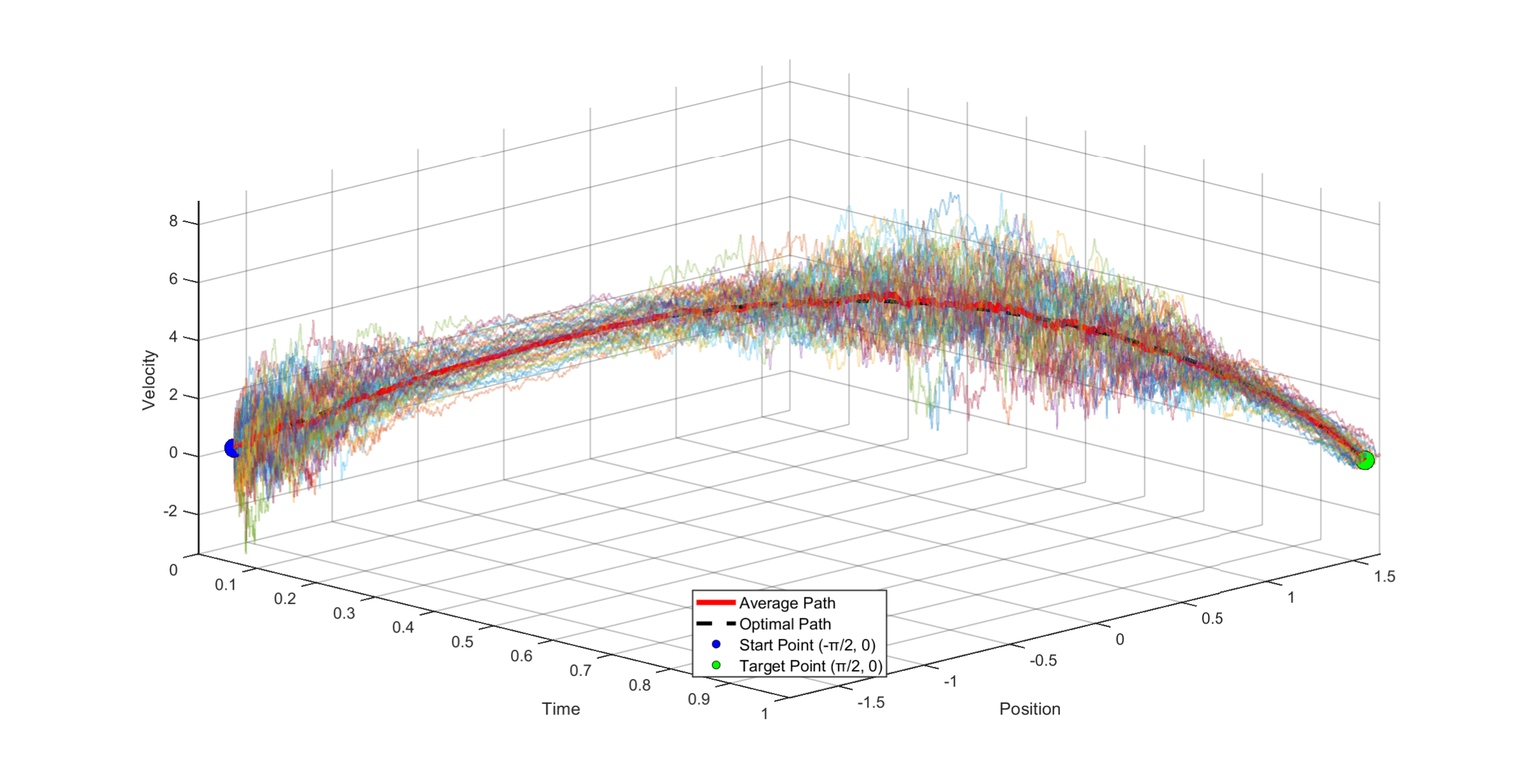}
    \caption{$H=0.3,\sigma_0=2,A=1.5,\omega=10$}
    \label{fig1a}
\end{figure}

\begin{figure}[htbp]
    \centering
    \includegraphics[width=0.9\textwidth]{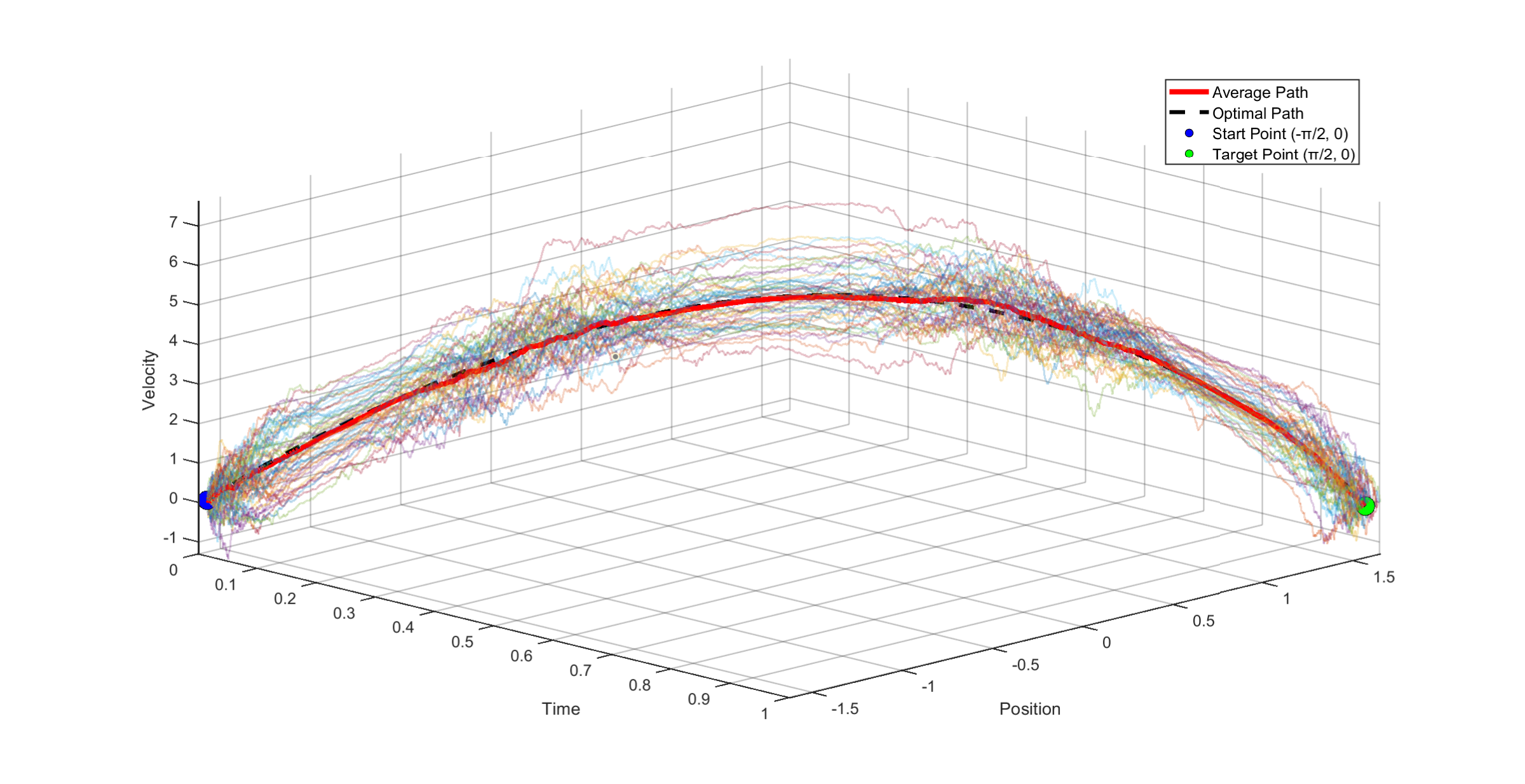}
    \caption{$H=0.5,\sigma_0=2,A=1.5,\omega=6\pi$}
    \label{fig1b}
\end{figure}

\begin{figure}[htbp]
    \centering
    \includegraphics[width=0.9\textwidth]{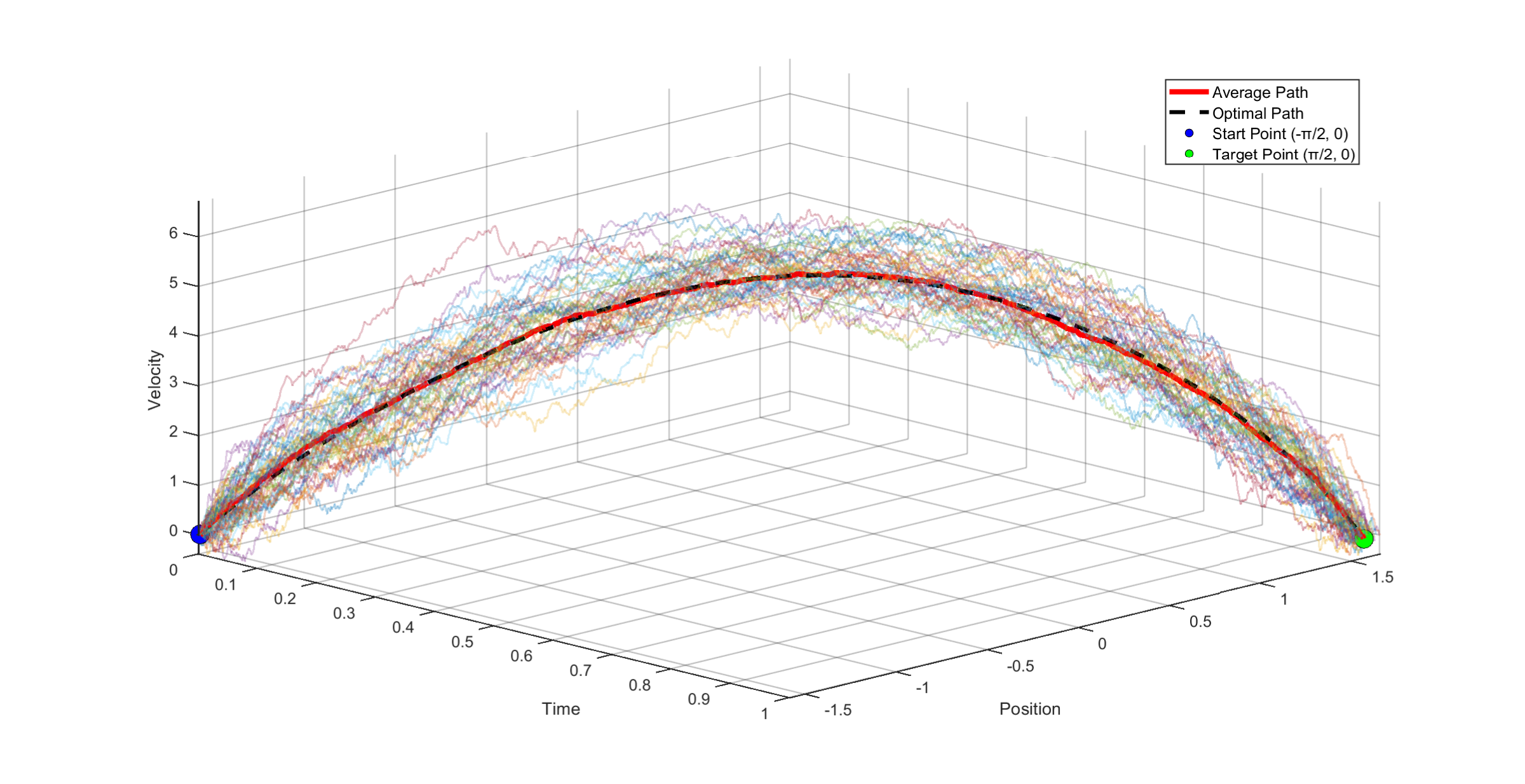}
    \caption{$H=0.51,\sigma_0=2,A=0.1,\omega=8\pi$}
    \label{fig1c}
\end{figure}

    The second example concerns a special Duffing equation, in which we investigate the most probable transition path between two equilibrium points.
 \begin{example}(Duffing Equation)
We now consider the equation
\begin{equation}  \label{equation 2}
    X''_t + \gamma X_t' + V'(X_t) = \sigma_t \xi_t^H .
\end{equation}
For simplicity and to observe its typical behavior, we choose 
$\gamma = 0.1$, $V(x) = \frac{1}{4}(x^4 - 2x^2)$, $\sigma_t = 3$, and $H = \frac{1}{2}$.
We consider the most probable transition path from $(-1, 0)$ to $(1, 0)$. 
For the system
\begin{equation*}
    \frac{d}{dt} \begin{bmatrix} X_t \\ Y_t \end{bmatrix} 
    = \begin{bmatrix} 
        Y_t \\ 
        -\gamma Y_t - V'(X_t)
    \end{bmatrix}
    = \begin{bmatrix}
        Y_t \\ 
        -\gamma Y_t - (X_t^2-1)X_t
    \end{bmatrix},
\end{equation*}
its linearization is
\begin{equation*}
    \frac{d}{dt} \begin{bmatrix} X_t \\ Y_t \end{bmatrix} 
    = \begin{bmatrix} 
        0 & 1 \\
        -V''(\bar{X}) & -\gamma 
    \end{bmatrix}
    \begin{bmatrix} X_t - \bar{X} \\ Y_t - \bar{Y} \end{bmatrix}.
\end{equation*}
The unperturbed system possesses equilibrium points at $(\pm 1, 0)$, and the addition of noise can induce transitions between these two states.
The corresponding OM functional is
\begin{equation*}
    J(\psi_s) = \int_0^1 \frac{\bigl( \psi_s'' + \gamma \psi_s' + V'(\psi_s) \bigr)^2}{\sigma^2} - \gamma  \, ds .
\end{equation*} 
For the autonomous case mentioned above, by exploiting the energy structure of the equation, its  OM functional can be reduced to:
\begin{align*}
    &\int_0^1 \bigl( \psi_s'' + \gamma \psi_s' + V'(\psi_s) \bigr)^2 ds \\
    &= \int_0^1 \bigl( \psi_s'' + V'(\psi_s) \bigr)^2 + 2\gamma \psi_s' \bigl( \psi_s'' + V'(\psi_s) \bigr) + \gamma^2 \psi_s'^2 \, ds \\
    &= \int_0^1 \bigl( \psi_s'' + V'(\psi_s) \bigr)^2 + \gamma^2 \psi_s'^2 \, ds 
       + \gamma \Bigl( \psi_1'^2 + V(\psi_1) - \psi_0'^2 - V(\psi_0) \Bigr) \\
    &= \int_0^1 \bigl( \psi_s'' + V'(\psi_s) \bigr)^2 + \gamma^2 \psi_s'^2 \, ds .
\end{align*}
Using variational methods, we derive the Euler--Lagrange equation:
\begin{equation*}
    \psi_s^{(4)} + \bigl( 2V''(\psi_s) - \gamma^2 \bigr) \psi_s'' 
    + V^{(3)}(\psi_s) \bigl( \psi_s' \bigr)^2 
    + V''(\psi_s) V'(\psi_s) = 0 .
\end{equation*}
Solutions of the Euler--Lagrange equation are compared with the mean trajectories of the stochastic system, thereby validating the accuracy of the theoretical results. See Fig.~\ref{example2}.   The numerical simulation for this case was conducted with the bvp4c algorithm.

Furthermore, we examine the case with a time-dependent parameter $\sigma_t = 2 + \sin(8\pi t)$ for $H = 0.3 ,  0.5$ and $0.55$. Under these  assumptions, due to the fractional-order calculus and the non-autonomous nature of the system, our Euler-Lagrange equation becomes considerably more complex. Consequently, the corresponding most probable path is obtained by minimizing the functional $J$ numerically via the interior-point method. Figures~\ref{example3} and~\ref{example4} present a comparison among the paths, the average path, and the most probable path for the singular case and the standard case. Since the main result in the regular case involve fractional derivatives, which pose significant challenges for numerical simulation, we present only the average path of the system here; see Fig~\ref{example5} and the related discussion. 

The numerical simulations reveal that the most probable path is affected not only by the Hurst parameter $H$ of the fractional Brownian motion but also by its time-modulated intensity $\sigma_t$. Moreover, the characteristic frequency of the noise manifests itself in the shape of the most probable path.

\end{example}

\begin{figure}[htbp]
    \centering
    
    \begin{subfigure}[b]{0.9\textwidth}
        \centering
        \includegraphics[width=\textwidth,height=0.4\textheight,keepaspectratio]{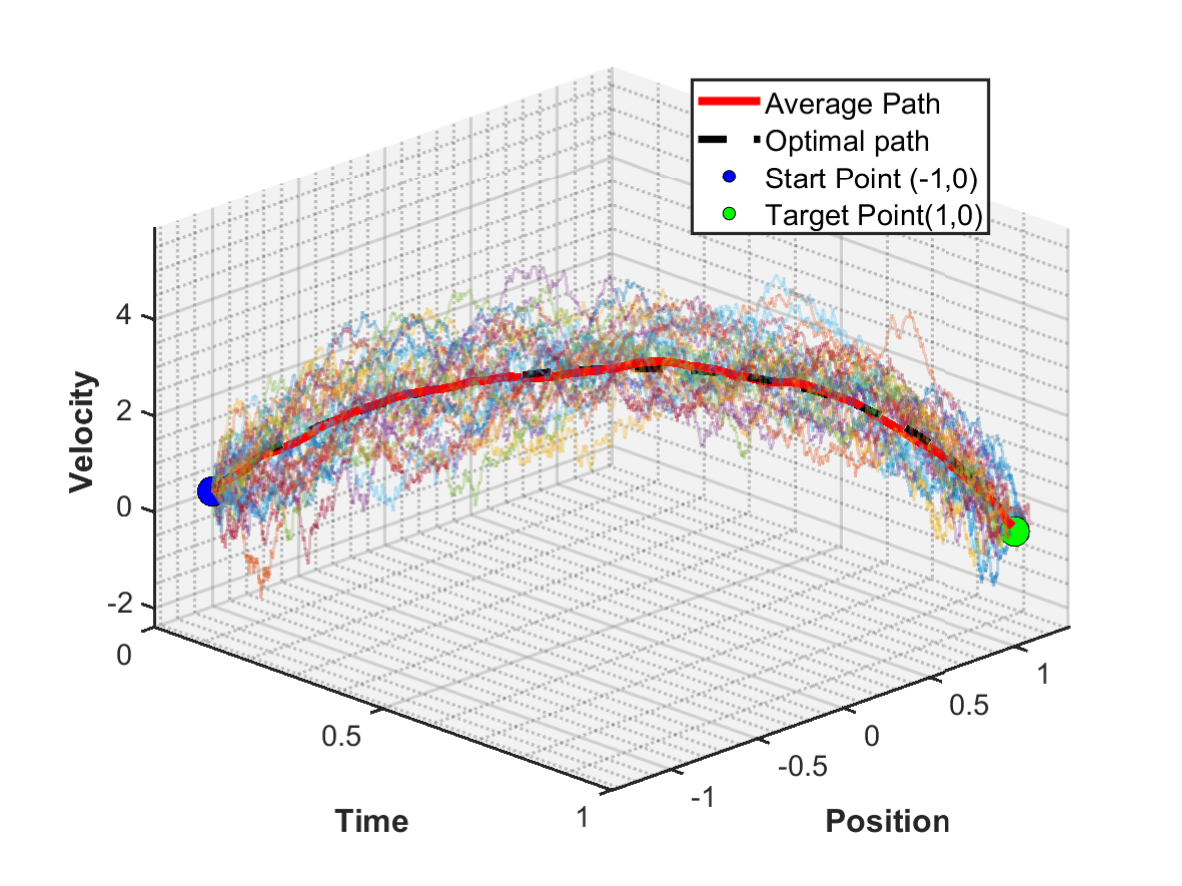}
        \caption{}
    \end{subfigure}
    
    \vspace{0.2cm}
    
    \begin{subfigure}[b]{0.9\textwidth}
        \centering
        \includegraphics[width=\textwidth,height=0.4\textheight,keepaspectratio]{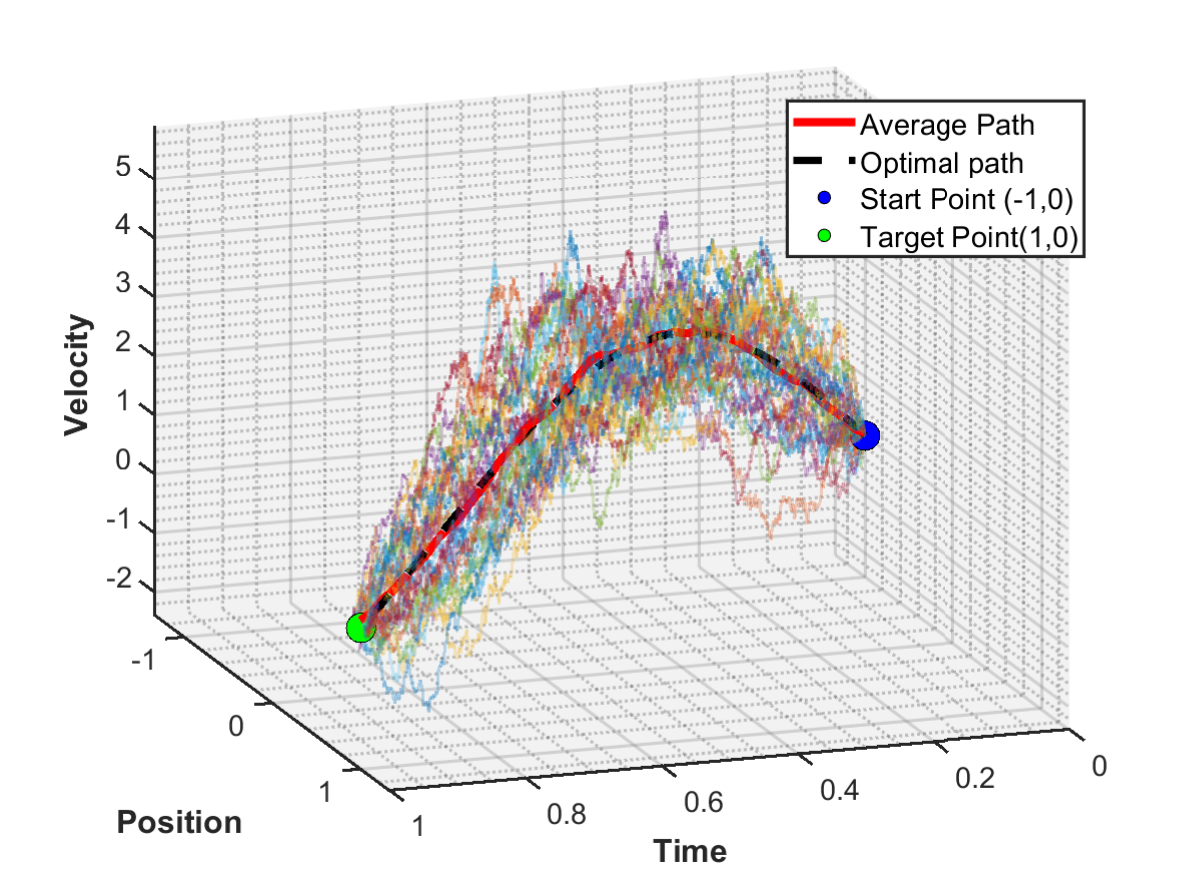}
        \caption{}
    \end{subfigure}
    
    \caption{The average path and the optimal path for \eqref{equation 2}}
    \label{example2}
\end{figure}

    \begin{figure}[htbp]
    \centering
    
    \begin{subfigure}[b]{0.9\textwidth}
        \centering
        \includegraphics[width=\textwidth]{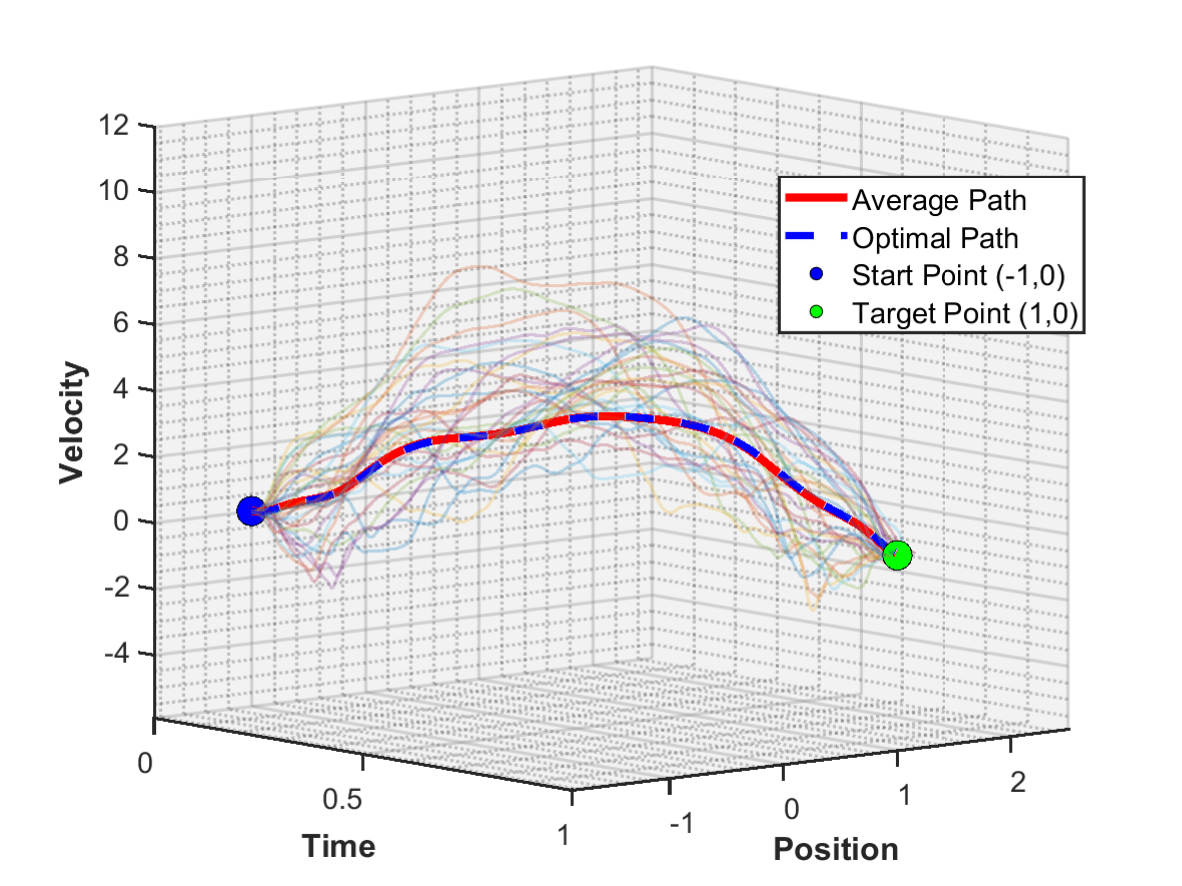}
        \caption{}
    \end{subfigure}
    
    \vspace{0.5cm}
    
    \begin{subfigure}[b]{0.9\textwidth}
        \centering
        \includegraphics[width=\textwidth]{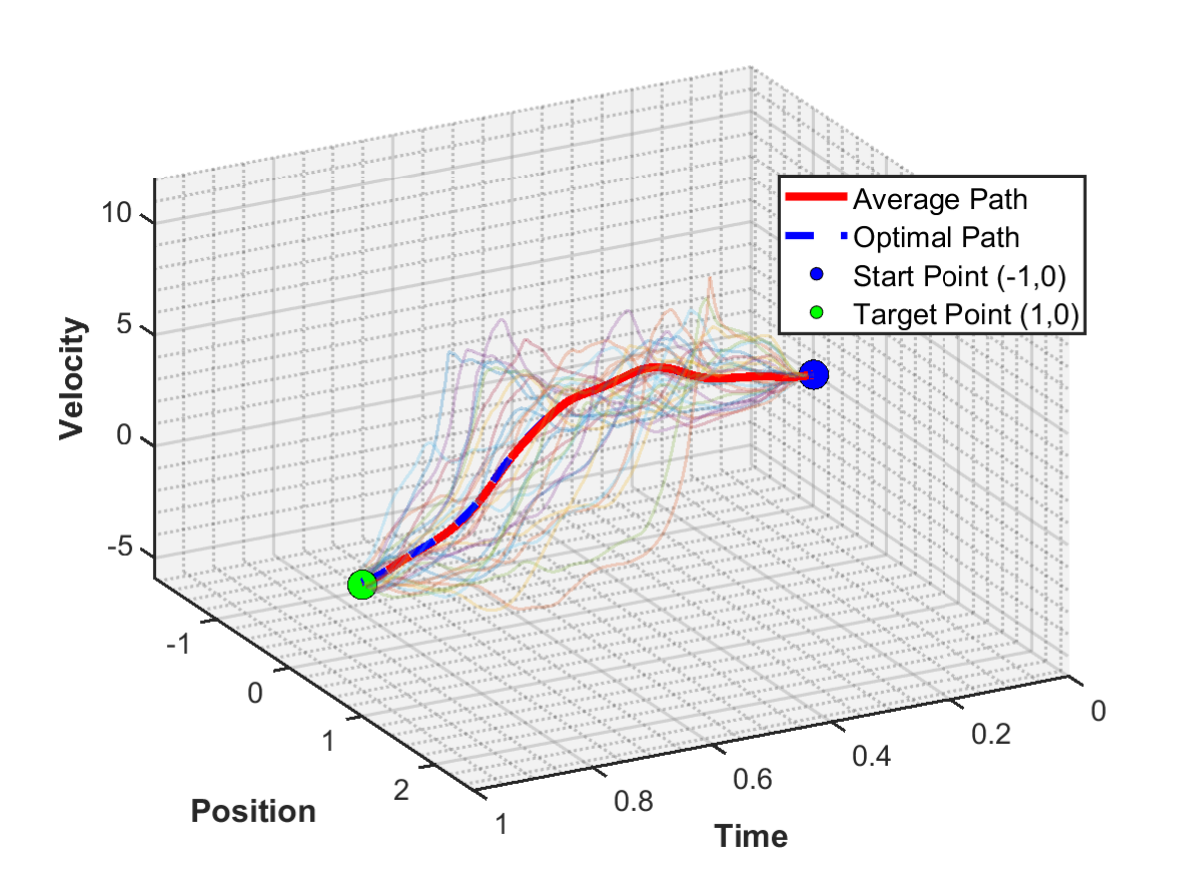}
        \caption{}
    \end{subfigure}
    
    \caption{The average path and the optimal path for \eqref{equation 2}, with $H=0.3, \sigma_t = 2 +  \cos(8\pi t)$ and $\gamma=0.1$}
    \label{example3}
\end{figure}

    \begin{figure}[htbp]
    \centering
    
    \begin{subfigure}[b]{0.9\textwidth}
        \centering
        \caption{}
        \includegraphics[width=\textwidth]{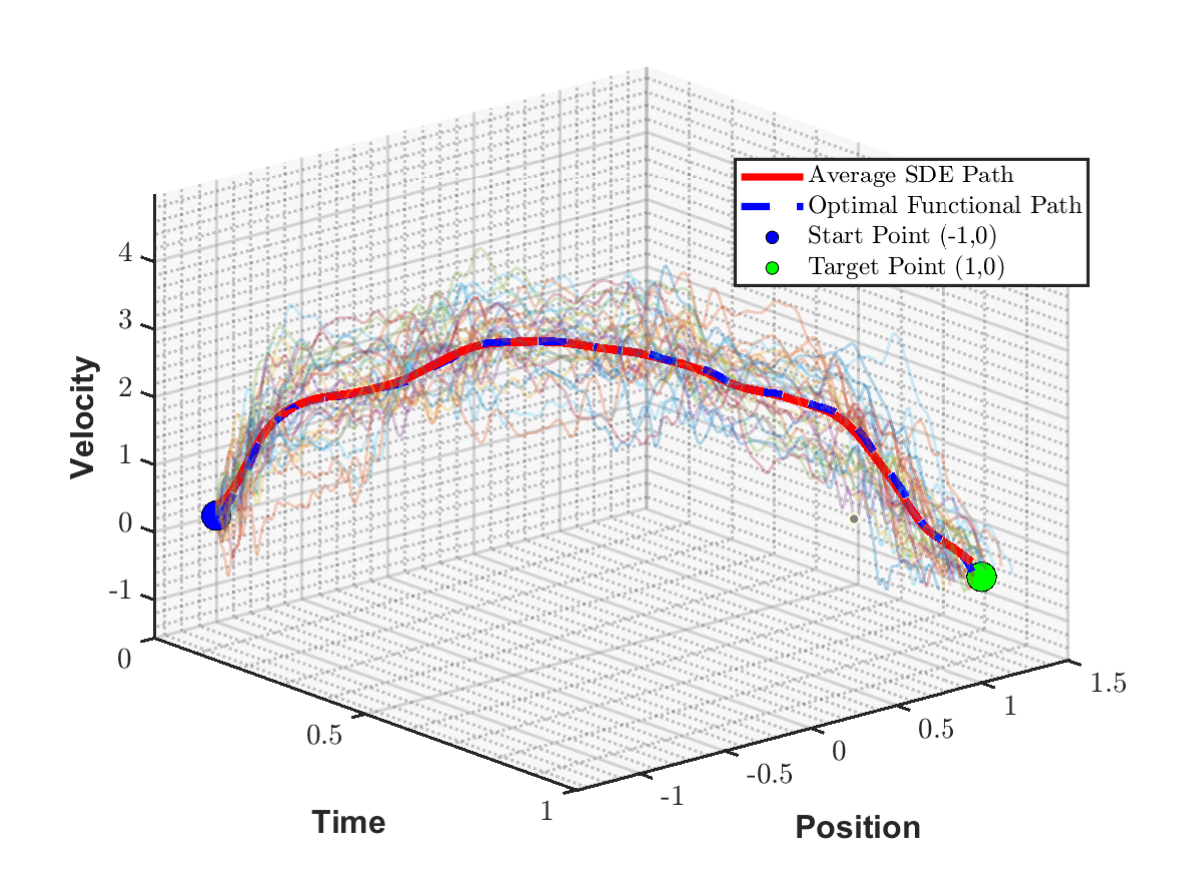}
    \end{subfigure}
    
    \vspace{0.5cm}
    
    \begin{subfigure}[b]{0.9\textwidth}
        \centering
        \caption{}
        \includegraphics[width=\textwidth]{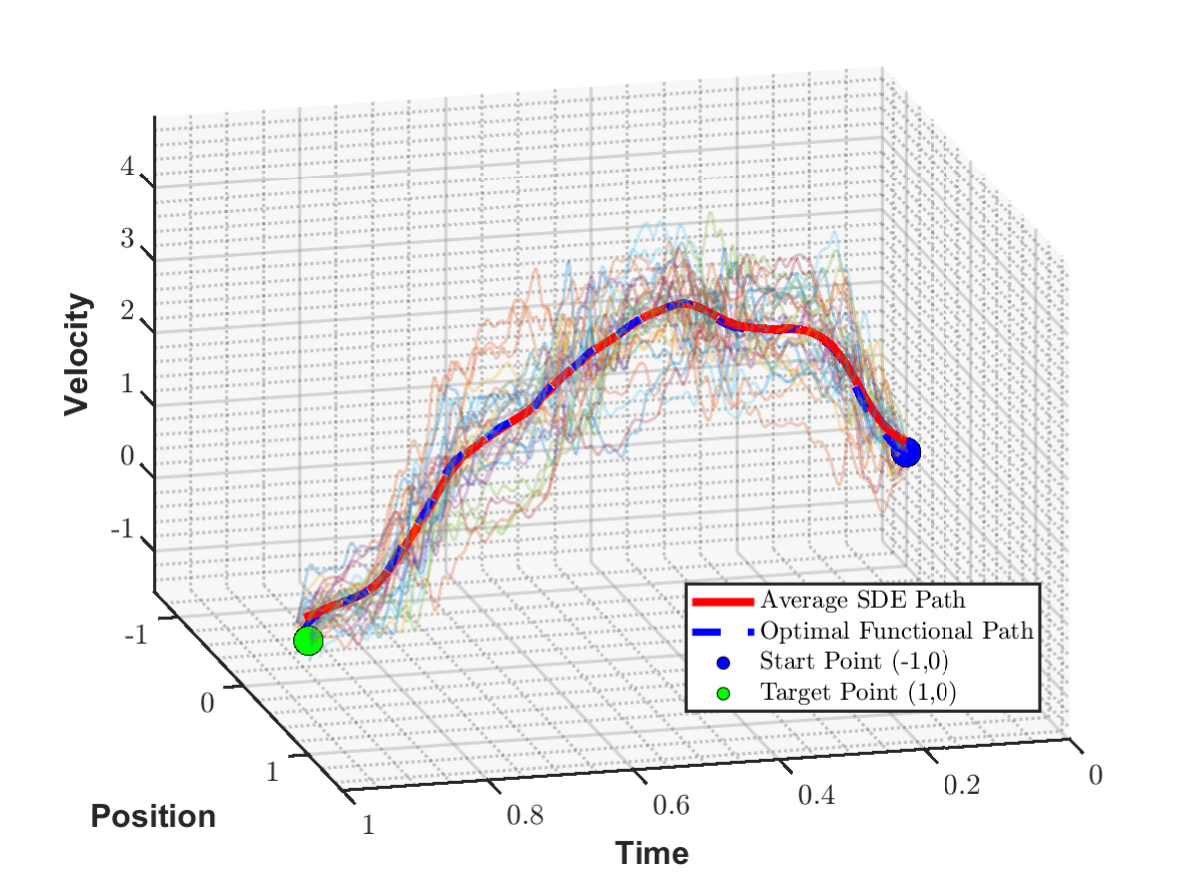}
    \end{subfigure}
    
    \caption{The average path and the optimal path for \eqref{equation 2}, with $H=0.5, \sigma_t = 2 +  \cos(8\pi t)$ and $\gamma=0.1$}
    \label{example4}
\end{figure}

\begin{figure}
    \centering
    \includegraphics[width=0.9\linewidth]{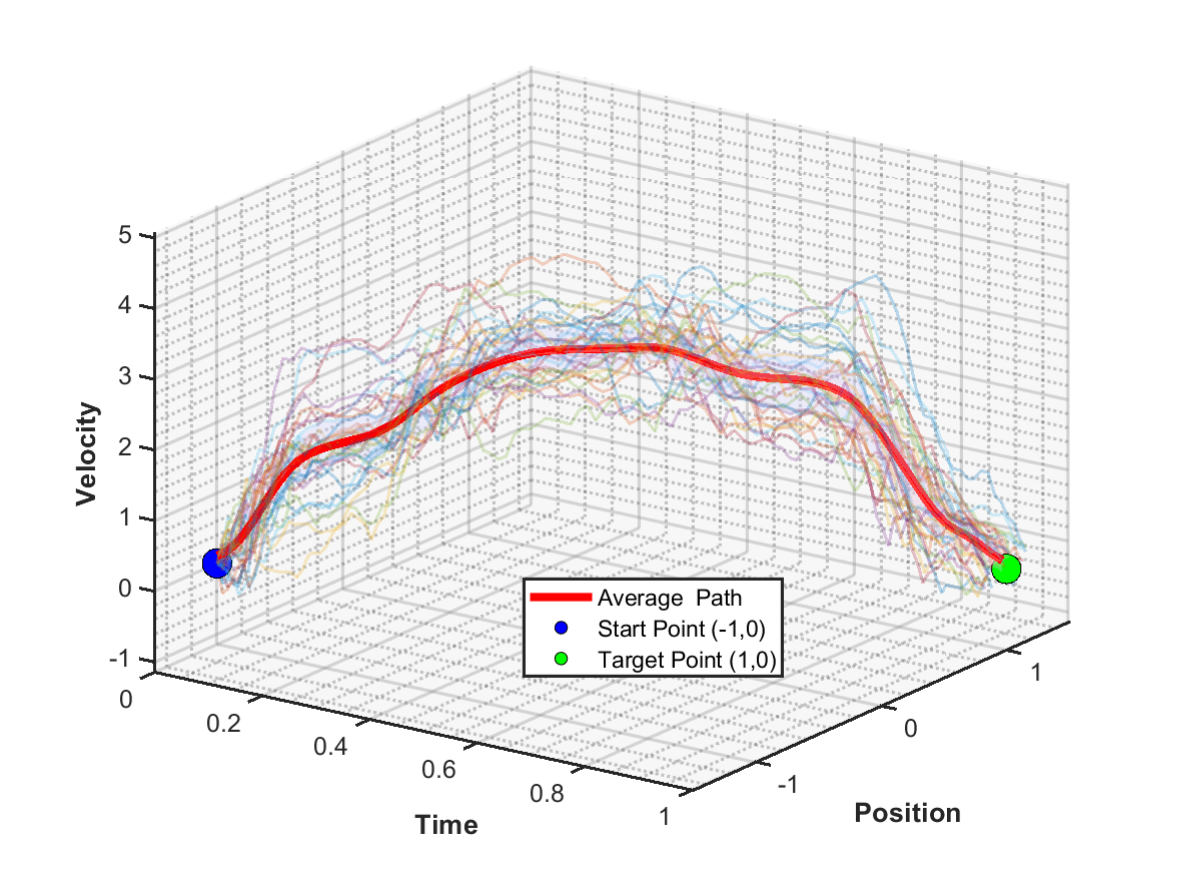}
    \caption{The average path  for \eqref{equation 2}, with $H=0.55, \sigma_t = 2 +  \cos(8\pi t)$ and $\gamma=0.1$}
    \label{example5}
\end{figure}

\section{Proof of the Main Results}

In this section, we will prove the main results and provide some tools necessary for the proofs. The main tools used in the proof in this paper remain the classical small ball probability estimates and the Girsanov theorem. We first present the small ball probability estimates, which are generalizations of classical results in small ball probability estimation, their proofs can be referred to in \cite{Nualart}.

	\begin{lemma}\label{estimate1}
If $0<\beta<H$, then 
\[
\lim _{\varepsilon \to 0} \mathbb{E}\left( \exp \left( \int_{0}^{1} h(s) \, d W_{s} \right) \mid \Vert\int_0^\cdot \sigma_s dB^H_s \Vert_\beta < \varepsilon \right) = 1,
\]
for all \( h \in L^{2}([0,1]) \).
\end{lemma}
    We recall that an operator $K : L^2([0,1]) \to L^2([0,1])$ is \emph{nuclear} iff 
\[
\sum_{n=1}^\infty \big| \langle K e_n, g_n \rangle \big| < \infty,
\]
for all orthonormal sequences $B=(e_n)$, $B'=(g_n)$ in $L^2([0,1])$.  
 The trace of a nuclear operator $K$ is defined by
\[
\operatorname{Tr} K = \sum_{n=1}^\infty \langle K e_n, e_n \rangle,
\]
for any orthonormal sequence $B=(e_n)$ in $L^2([0,1])$.  
The definition is independent of the sequence we have chosen.

Given a symmetric function $f \in L^2([0,1]^2)$, the Hilbert--Schmidt operator 
$K(f) : L^2 \to L^2$ associated with $f$ is defined by
\begin{equation}\label{eq:Kf}
  (K(f)h)(t) = \int_0^t f(t,u) h(u)\,du.
\end{equation}
The operator $K(f)$ is nuclear iff
\[
\sum_{n=1}^\infty \big|\langle K e_n, e_n \rangle \big| < \infty
\quad \text{for all orthonormal sequences } (e_n) \subset L^2([0,1]).
\]

If $f$ is continuous and $K(f)$ is nuclear, we can compute its trace as follows:
\[
\operatorname{Tr} f := \operatorname{Tr} K(f) = \int_0^1 f(s,s)\,ds.
\]

	\begin{lemma}\label{estimate2}
Let \( f \) be a symmetric function in \( L^{2}([0,1]^{2}) \) . If \( K(f) \) is nuclear  and $0<\beta<H$, then 
\[
\lim _{\varepsilon \to 0} E\left( \exp \left( \int_{0}^{1} \int_{0}^{1} f(s, t) \, d W_{s} d W_{t} \right) \mid \Vert\int_0^\cdot \sigma_s dB^H_s \Vert_\beta < \varepsilon \right) = e^{-\text{Tr}(f)}.
\]
\end{lemma}

\begin{lemma}
Let $ \sigma $ satisfy Assumption \((A)\), let $U_t^H=\int_0^t \sigma_vdB^H_v$.
Then for any $0 < \beta < H$, there exists a constant $M_\beta > 0$ such that
\begin{equation}
\mathbb{P}\Bigg(
\sup_{0 \leq s \neq t \leq 1} \frac{|U^H_t - U^H_s|}{|t-s|^\beta} \leq \varepsilon
\Bigg) 
\geq \exp\Big(\varepsilon^{-\frac{1}{H-\beta}} M_\beta \Big). \label{smbp}
\end{equation}
\end{lemma}

\begin{lemma}[Girsanov's Theorem, {\cite[Theorem~8.6]{Oksendal2014}}]\label{lem:girsanov}
Let $(\Omega,\mathcal{F},\{\mathcal{F}_t\}_{t\ge0},\mathbb{P})$ be a filtered probability space, and let 
$W=\{W_t\}_{t\ge0}$ be a $d$-dimensional Brownian motion. 
Suppose that $u=\{u_t\}_{t\ge0}$ is an $\{\mathcal{F}_t\}$-adapted process such that the following Novikov condition holds:
\begin{equation}
    \mathbb{E}\left[\exp\left(\tfrac{1}{2}\int_0^T |u_s|^2\,ds\right)\right] < \infty      \label{novikov}
\end{equation}
for every $T>0$. Define
\[
Z_t = \exp\left(-\int_0^t u_s \cdot dW_s - \tfrac{1}{2}\int_0^t |u_s|^2\,ds\right), 
\qquad t \in [0,T].
\]
Then $\{Z_t\}$ is a martingale under $\mathbb{P}$. If we define a new probability measure 
$\mathbb{Q}$ on $(\Omega,\mathcal{F}_T)$ by
\[
\frac{d\mathbb{Q}}{d\mathbb{P}}\Big|_{\mathcal{F}_T} = Z_T,
\]
then the process
\[
\widetilde{W}_t = W_t + \int_0^t u_s\,ds, \qquad t \in [0,T]
\]
is a $d$-dimensional Brownian motion with respect to $\mathbb{Q}$.
\end{lemma}

   \begin{lemma}[ Stochastic Fubini theorem,{\cite[Theorem 2.2]{Veraar01082012}}]\label{fubini}
   Let \((X, \Sigma, \mu)\) be a \(\sigma\)-finite measure space. Let \(\varphi: X \times [0,T] \times \Omega \to \mathbb{R}\) be progressively measurable, and suppose that for almost every \(\omega \in \Omega\),
\[
\int_X \left( \int_0^T |\varphi(x,t,\omega)|^2 \, dt \right)^{1/2} \, d\mu(x) < \infty.
\]
   	Then for almost all $\omega\in \Omega$ and for all $t\in [0,T]$, we have
   	\begin{equation*}
   		\int_X \int_0^T \varphi(x,t,\omega) dW_t d\mu(x)=\int_0^T \int_X \varphi(x,t,\omega) d\mu(x) dW_t .
   	\end{equation*}
   	  \end{lemma}

 \begin{lemma}[\cite{inequality}, Lemma 2.1]
    Let $(\Omega,\mathcal{F},P)$ be a probability triple and let $\{M_t\}_{t\geq 0}$ be a locally square integrable martingale w.r.t. the filtration 
    $\{\mathcal{F}_t\}_{t\geq 0}$, $M_0=0.$ 
    Let \(\langle M\rangle_t\) denote the quadratic variation of \(\{M_t\}\). Suppose that $|\Delta M_t|=|M_t-M_{t^-}|\leq K$ for all $t>0$, $0\leq K<\infty.$
    Then for each $a>0$, $b>0$,
    \begin{equation*}
        P(M_t\geq a \ \text{and} \ \langle M\rangle_t \leq b^2 \ \text{for some} \ t)
        \leq \exp\left(-\frac{a^2}{2(aK+b^2)}\right).
    \end{equation*}
    \label{exponential_ineq} 
\end{lemma}

First, the OM functional in \eqref{omf} will be simplified using the Girsanov transformation. Then, we will compute the functional for different ranges of \(H\) (i.e., \(0 < H < \frac{1}{2}\), \(\frac{1}{2} < H < 1\), and \(H = \frac{1}{2}\)).

Let 
\begin{equation}
    \tilde{Y_t}=\phi_t+\int_0^t \sigma_s\  dB^H_s,\quad \tilde{X}_t=x_0+\int_0^t  \tilde{Y}_s  \ ds. 
\end{equation}
We define new processes
\begin{equation}
     u_s=  \dot{\phi}_s 
        - (K_H^\sigma)^{-1}\!\left( \int_0^{\cdot} f_u(\tilde{X}_u,\tilde{Y}_u) \, du \right)(s) 
\end{equation}
and
\begin{equation}
    \tilde{W}_t 
    = W_t + \int_0^t 
   u_s ds. 
\end{equation}

Under the assumptions of this paper, the Novikov condition \eqref{novikov} is satisfied (a proof can be found in\cite{NUALART2002103} ). By Lemma~\ref{lem:girsanov}, 
there exists a probability measure $\tilde{\mathbb{P}}$, absolutely continuous with respect to $\mathbb{P}$,  
such that $\tilde{W}_t$ is a Brownian motion under $\tilde{\mathbb{P}}$.  
The corresponding fractional Brownian motion associated with $\tilde{W}_t$ is given by
\begin{equation*}
    \tilde{B}^H_t 
    = B^H_t 
    + \int_0^t \sigma_s^{-1} \, d\phi_s
    - \int_0^t \sigma_s^{-1} f_s(\tilde{X}_s,\tilde{Y_s}) \, ds.
\end{equation*}
Taking the differential of $\tilde{X}_t$ and $\tilde{Y}_t$, we have
\begin{equation*}
\begin{cases}
d\tilde{X}_t = \tilde{Y}_t dt, \\
d\tilde{Y}_t = f_t(\tilde{X}_t, \tilde{Y}_t) dt + \sigma_t d\tilde{B}^H_t.
\end{cases}
\end{equation*} 
This implies that $(\tilde{X}, \tilde{B}^H)$ is the strong solution of \eqref{fir} under the probability measure $\tilde{\mathbb{P}}$.
 Thus, the OM functional can be simplified as follows:  
\[
\begin{aligned}
&\frac{\mathbb{P}(\|X - \psi\|_{\beta+1} \leq \varepsilon)}{\mathbb{P}\left( \left\| \int_0^\cdot \sigma_u \, dB^H_u \right\|_\beta \leq \varepsilon \right)} \\
=& \frac{\tilde{\mathbb{P}}(\| \tilde{Y} - \phi \|_{\beta} \leq \varepsilon)}{\mathbb{P}\left( \left\| \int_0^\cdot \sigma_u \, dB^H_u \right\|_{\beta} \leq \varepsilon \right)} \\
=& \frac{\tilde{\mathbb{P}}\left( \left\| \int_0^\cdot \sigma_u \, dB^H_u \right\|_{\beta} \leq \varepsilon \right)}{\mathbb{P}\left( \left\| \int_0^\cdot \sigma_u \, dB^H_u \right\|_{\beta} \leq \varepsilon \right)} \\
=& \mathbb{E} \left[ \exp \left( -\int_0^1 u_s \, dW_s - \frac{1}{2} \int_0^1 u_s^2 \, ds \right) \,\middle|\, \left\| \int_0^\cdot \sigma_u \, dB^H_u \right\|_{\beta} \leq \varepsilon  \right].
\end{aligned}
\]

\subsection{Singular Case}
In this section we will compute the OM functional for $\frac{1}{4}< H<\frac{1}{2}.$
\begin{theorem} \label{result1} Let $X$ be the solution of \eqref{fir}, and assume that the coefficients satisfy Assumption \((A)\). 
For any $\phi=\psi' \in K_H^\sigma(L^2([0,1]))$ with $(\psi_0,\phi_0) = (x_0,y_0)$, when $ \tfrac{1}{4} < H < \tfrac{1}{2}$, 
the OM functional of $X_t$ with respect to the norms 
$\Vert \cdot \Vert_{\beta+1}$, where $H-\tfrac{1}{2}<\beta<H-\tfrac{1}{4}$,  can be expressed as  
\begin{equation*}
J(\psi)=-\dfrac{1}{2}\displaystyle\int_0^1 
\Big(
\dot{\phi}_s
- s^{-\alpha} I_{0^+}^\alpha s^{\alpha} 
\sigma_s^{-1} f_s(\psi_s,\phi_s)
\Big)^2 
+ d_H \partial_y f_s(\psi_s,\phi_s) \, ds,
\end{equation*}
where
\begin{equation*}
d_H = \sqrt{\frac{2H \,\Gamma\!\left(\tfrac{1}{2}+H\right)\Gamma\!\left(\tfrac{3}{2}-H\right)}{\Gamma(2-2H)}}.
\end{equation*}

\end{theorem}

\begin{proof}
For the parts identical to the classical setting, we omit the detailed proofs; further details can be found in \cite{Liu2026,Nualart}.
Having applied the Girsanov theorem, the remainder of the proof proceeds using standard arguments. By the Girsanov theorem, we obtain
\begin{align*}
\mathbb{P}(\Vert X-\psi\Vert_{\beta+1} \leq \varepsilon)
&= \mathbb{E}\Bigg[\exp\Bigg( -\int_0^1 u_s \, dW_s - \frac{1}{2} \int_0^1 u_s^2 \, ds\Bigg)
  I_{\big\{\Vert \int_0^\cdot \sigma_s \, dB^H_s\Vert_\beta \leq \varepsilon\big\}}\Bigg] \\
&= \mathbb{E}\Big(\exp(A_1+A_2)\,
  I_{\big\{\Vert \int_0^\cdot \sigma_s \, dB^H_s\Vert_\beta \leq \varepsilon\big\}}\Big),
\end{align*}
where
\begin{align*}
A_1 &= -\int_0^1 \Big(\dot{\phi}_s - s^{-\alpha} I_{0^+}^\alpha\big(s^{\alpha} \sigma_s^{-1}
      f_s(\tilde{X}_s,\tilde{Y}_s)\big)\Big)\, dW_s, \\
A_2 &= -\frac{1}{2}\int_0^1 \Big(\dot{\phi}_s - s^{-\alpha} I_{0^+}^\alpha\big(s^{\alpha} \sigma_s^{-1}
      f_s(\tilde{X}_s,\tilde{Y}_s)\big)\Big)^2 ds.
\end{align*}
Moreover, by
\begin{equation*}
\tilde{Y}_t = \phi_t + \int_0^t \sigma_s \, dB^H_s, 
\qquad 
\tilde{X}_t = \psi_t + \int_0^t ds \int_0^s \sigma_u \, dB^H_u,
\end{equation*}
it is straightforward to see that 
\[
(\tilde{X}_t,\tilde{Y}_t) \to (\psi_t,\phi_t)
\quad \text{as } \left\| \int_0^\cdot \sigma_u \, dB^H_u \right\|  \to 0.
\]
That implies
\[
\lim_{\varepsilon\to 0}
\mathbb{E}\Big(\exp(A_2)\,\Big|\,
\left\| \int_0^\cdot \sigma_u \, dB^H_u \right\| \leq \varepsilon \Big) = 1.
\]

 Next, it suffices to consider the limit
\begin{equation}
 \mathbb{E}\Big(\exp\left(\int_0^1  s^{-\alpha} I_{0^+}^\alpha s^{\alpha} \sigma_s^{-1}
      f_s(\tilde{X}_s,\tilde{Y}_s)\, dW_s\right)\,\Big|\,
\left\| \int_0^\cdot \sigma_u \, dB^H_u \right\| \leq \varepsilon \Big) .\label{taylor1}
\end{equation}
We perform a Taylor expansion in \eqref{taylor1}:
\begin{align*}
    s^{-\alpha} I_{0^+}^\alpha s^{\alpha} \sigma_s^{-1}
      f_s(\tilde{X}_s,\tilde{Y}_s) &= s^{-\alpha} I_{0^+}^\alpha s^{\alpha} \sigma_s^{-1}\Big(f_s(\psi_s, \phi_s) \\
     & \qquad + \partial_x f_s(\psi_s,\phi_s)\int_0^s dv \int_0^v \sigma_u dB^H_u \\
      & \qquad + \partial_y f_s(\psi_s,\phi_s)\int_0^s \sigma_u dB^H_u \Big) + R_s.
\end{align*}
However, unlike the classical case, an additional term involving the stochastic integral of the noise appears. Since integral inequalities do not hold in the context of stochastic integration, the analysis of its limiting behavior becomes non-trivial.

Here we only consider the term arising from the degeneracy of the noise compared to the classical equation:  
\begin{equation*}
    \lim_{\varepsilon\to 0}\mathbb{E} \left[ \exp \left( \int_0^1 \partial_x f_s(\psi_s,\phi_s)\int_0^s dv \int_0^v \sigma_u \,dB^H_u  \, dW_s \right) \,\middle|\, \left\| \int_0^\cdot \sigma_u \, dB^H_u \right\| \leq \varepsilon  \right].
\end{equation*}
Here
\begin{align*}
    &s^{-\alpha}I_{0^+}^\alpha s^{\alpha} \sigma_s^{-1}\partial_x f_s(\psi_s,\phi_s)\int_0^s dv \int_0^v \sigma_u dB^H_u  \\
    =& s^{-\alpha}I^\alpha_{0^+}s^{\alpha} \sigma_s^{-1}\partial_x f_s(\psi_s,\phi_s)\int_0^s \int_0^u \sigma_v dB^H_v du \\
    =& C s^{-\alpha}\int_0^s (s-r)^{\alpha-1} r^{\alpha} \sigma_r^{-1}\partial_x f_r(\psi_r,\phi_r)\left( \int_0^r   \int_0^u \sigma_v \, dB^H_v   du   \right)dr.
\end{align*}
where
\begin{align*}
    \int_0^u \sigma_v dB^H_v &= \int_0^u v^{\alpha} D_{u^-}^\alpha  v^{-\alpha}\sigma_v dW_v, \\
    &= C\int_0^u v^\alpha\left(  \frac{v^{-\alpha}\sigma_v}{(u-v)^\alpha} 
	- \alpha \int_v^u \frac{x^{-\alpha}\sigma_x - v^{-\alpha}\sigma_v}{(x-v)^{\alpha+1}} dx \right) dW_v.
\end{align*}
   We have
\begin{align*}
    \int_0^1  s^{-\alpha}I_{0^+}^\alpha s^{\alpha} \sigma_s^{-1}\partial_x f_s(\psi_s,\phi_s)(\tilde{X}_s-\psi_s) \, dW_s
    &= \int_0^1 \int_0^s g(s,v) \, dW_v \, dW_s \\
    &= \int_0^1 \int_0^s \left( g_1(s,v) + g_2(s,v) \right) \, dW_v \, dW_s,
\end{align*}
where
\begin{align*}
    g_1(s,v) &= C \int_v^s dr \int_v^r du \, (s-r)^{\alpha-1} r^\alpha \sigma_r^{-1} \partial_x f_r(\psi_r,\phi_r) \sigma_v (u-v)^{-\alpha}, \\
    g_2(s,v) &= C \int_v^s dr \int_v^r du \int_v^u dx \, (s-r)^{\alpha-1} r^\alpha \sigma_r^{-1} \partial_x f_r(\psi_s,\phi_r) v^\alpha \frac{x^{-\alpha}\sigma_x - v^{-\alpha}\sigma_v}{(x-v)^{\alpha+1}}.
\end{align*}

For \(g_1(s,v)\), we have
\begin{align*}
    |g_1(s,v)| &\leq C \int_v^s dr \int_v^r du \, (s-r)^{\alpha-1} r^\alpha (u-v)^{-\alpha} \\
    &\leq C \int_v^s (s-r)^{\alpha-1} r^\alpha (r-v)^{1-\alpha} \, dr \\
    &\leq C s^\alpha (s-v) \int_0^1 (1-y)^{\alpha-1} y^{1-\alpha} \, dy,
\end{align*}
thus, \(g_1(s,s) = 0\).

For \(g_2(s,v)\), we have
\begin{align*}
    |g_2(s,v)| &\leq C \int_v^s dr \int_v^r du \int_v^u dx \, (s-r)^{\alpha-1} r^\alpha v^\alpha \frac{x^{-\alpha}\sigma_x - v^{-\alpha}\sigma_v}{(x-v)^{\alpha+1}} \\
    &\leq C \int_v^s dr \int_v^r du \, (s-r)^{\alpha-1} r^\alpha v^\alpha \left( \int_v^u \frac{x^{-\alpha} - v^{-\alpha}}{(x-v)^{\alpha+1}} \, dx + (x-v)^{-\alpha} \right) \\
    &\leq C \int_v^s dr \int_v^r du\\
    &\qquad\cdot\, (s-r)^{\alpha-1} r^\alpha v^\alpha \left( v^{-2\alpha} \int_1^\infty (1-y^\alpha)y^{-\alpha} (y-1)^{-(\alpha+1)}  \, dy + (u-v)^{1-\alpha} \right) \\
    &\leq C \int_v^s dr \int_v^r du \, (s-r)^{\alpha-1} r^\alpha v^\alpha \left( v^{-2\alpha} + (u-v)^{1-\alpha} \right) \\
    &\leq C \int_v^s dr (s-r)^{\alpha-1} r^\alpha v^{-\alpha} (r-v) + (s-r)^{\alpha-1} r^\alpha v^\alpha (r-v)^{2-\alpha} \\
    &\leq C \left( s^\alpha v^{-\alpha} (s-v)^{\alpha+1} \int_0^1 (1-y)^{\alpha-1} y \, dy + (sv)^\alpha (s-v)^2 \right).
\end{align*}
Thus, \(g(s,s) = 0\).

According to Lemma~\ref{estimate2}, we have
\begin{equation*}
    \lim_{\varepsilon\to 0} \mathbb{E} \left[ \exp \left( \int_0^1 \partial_x f_s(\psi_s,\phi_s) \int_0^s dv \int_0^v \sigma_u dB^H_u \, dW_s \right) \,\middle|\, \left\| \int_0^\cdot \sigma_u \, dB^H_u \right\| \leq \varepsilon \right] = 0.
\end{equation*}

\end{proof}

    \subsection{Regular Case }
In this section we will compute the OM functional for $ \frac{1}{2}<H<1.$
\begin{theorem} \label{result2}  Let $X$ be the solution of \eqref{fir}, and assume that the coefficients satisfy Assumption \((A)\).
For any $\phi=\psi' \in K_H^\sigma(L^2([0,1]))$ with $(\psi_0,\phi_0) =(x_0,y_0)$, when $ \tfrac{1}{2} < H < 1$, 
the OM functional of $X_t$ with respect to the norms 
$\Vert \cdot \Vert_{\beta+1}$, where $H-\tfrac{1}{2}<\beta<H-\tfrac{1}{4}$, 
 can be expressed as 
	\begin{equation*}
	J(\psi)=-\frac{1}{2} \int_0^1 \left(\dot{\phi}_s-s^{\alpha}D_{0^+}^\alpha s^{-\alpha}  \sigma_s^{-1} f_s\left(\psi_s,\phi_s\right)\right)^2+d_H \partial_yf_s(\psi_s,\phi_s)ds,
\end{equation*}
where
\begin{equation*}
    d_H= \sqrt{\frac{2H\Gamma(H+1/2)\Gamma(3/2-H)}{\Gamma(2-2H)}}.
\end{equation*}

\end{theorem}
\begin{proof}
  The proof in the regular case follows the same strategy.  The  difference lies in the computations involving fractional integrals and fractional derivatives. Hence, we proceed directly to the main part of the argument
\begin{align*}
&\int_0^1 s^{\alpha}D^\alpha_{0^+} s^{-\alpha} \sigma_s^{-1} \partial_x f_s(\psi_s, \phi_s)(\tilde{X}_s - \psi_s) \, dW_s\\
    =& \int_0^1 s^\alpha \Bigg[ s^{-2\alpha} \sigma_s^{-1} \partial_x f_s(\psi_s, \phi_s) \int_0^s du \int_0^u dB^H_v \sigma_v \\
    &\qquad\qquad + \alpha \int_0^s dr\ (s - r)^{-(\alpha + 1)} s^{-\alpha} \sigma_s^{-1} \partial_x f_s(\psi_s, \phi_s) \int_0^s du \int_0^u dB^H_v \sigma_v  \\
     &\qquad\qquad - \alpha \int_0^s (s - r)^{-(\alpha + 1)} \, dr\ r^{-\alpha} \sigma_r^{-1} \partial_x f_r(\psi_r, \phi_r) \int_0^r du \int_0^u dB^H_v \sigma_v  \Bigg] dW_s \\
    =& \int_0^1 \int_0^s g_1(s,v) \, dW_v \, dW_s + I_1,
\end{align*}
where
\begin{align*}
    g_1(s,v) &= \int_v^s du \int_v^u (x-v)^{\alpha-1} x^\alpha \sigma_x v^{-\alpha} s^{-\alpha} \sigma_s^{-1} \partial_x f_s(\psi_s, \phi_s) \, dx
\end{align*}
and
\begin{align*}
      &I_1\\
      &= \alpha \int_0^1 s^\alpha \, dW_s \\
    &\cdot\Bigg( \int_0^s \frac{s^{-\alpha} \sigma_s^{-1} \partial_x f_s(\psi_s, \phi_s) \int_0^s du \int_0^u dB^H_v \sigma_v - r^{-\alpha} \sigma_r^{-1} \partial_x f_r(\psi_r, \phi_r) \int_0^r du \int_0^u dB^H_v \sigma_v}{(s - r)^{\alpha + 1}} \, dr\Bigg).
\end{align*}
We have
\begin{align*}
    |g_1(s,v)| &\leq C \int_v^s du \, s^{-\alpha} v^{-\alpha} \int_v^u (x-v)^{\alpha-1} x^\alpha \, dx \\
    &\leq C \int_v^s du \, s^{-\alpha} v^{-\alpha} u^\alpha (u-v)^\alpha \, du \\
    &\leq C v^{-\alpha} (s-v)^{\alpha+1},
\end{align*}
which means that \( g_1(s,s) = 0 \).

Next, we divide \( I_1 \) into three parts: \( K_1 \), \( K_2 \), and \( K_3 \). Then
\begin{align*}
    K_1 &= C \int_0^1 s^\alpha \, dW_s \\
    &\quad \cdot \int_0^s \frac{s^{-\alpha} \sigma_s^{-1} \partial_x f_s(\psi_s, \phi_s) \int_0^s du \int_0^u dB^H_v \sigma_v - r^{-\alpha} \sigma_r^{-1} \partial_x f_r(\psi_r, \phi_r) \int_0^r du \int_0^u dB^H_v \sigma_v}{(s - r)^{\alpha + 1}} \, dr \\
    &= C \int_0^1 s^\alpha \partial_x f_s(\psi_s, \phi_s) \left( \int_0^s du \int_0^u \sigma_v \, dB^H_v \right) \left( \int_0^s \frac{s^{-\alpha} - r^{-\alpha}}{(s - r)^{\alpha + 1}} \, dr \right) dW_s \\
    &= C \int_0^1 s^\alpha \sigma_s^{-1} \partial_x f_s(\psi_s, \phi_s) \\
    &\qquad\qquad \cdot \left( \int_0^s du \int_0^u v^{-\alpha} \, dW_v \int_v^u (x-v)^{\alpha-1} x^\alpha \sigma_x \, dx \right) \left( \int_0^s \frac{s^{-\alpha} - r^{-\alpha}}{(s - r)^{\alpha + 1}} \, dr \right) dW_s \\
    &= C \int_0^1 dW_s \int_0^s dW_v \, g_2(s,v),
\end{align*}
where
\[
    g_2(s,v) = \int_v^s s^\alpha \sigma_s^{-1} \partial_x f_s(\psi_s, \phi_s) v^{-\alpha} \left( \int_v^u (x-v)^{\alpha-1} x^\alpha \sigma_x \, dx \right) \left( \int_0^s \frac{s^{-\alpha} - r^{-\alpha}}{(s - r)^{\alpha + 1}} \, dr \right) du.
\]

    We have
    \begin{align*}
        |g_2(s,v)|&\leq C \int_v^s s^{-\alpha} v^{-\alpha}  u^\alpha (u-v)^\alpha \left( \int_0^1 (1-y^{-\alpha})(1-y)^{-\alpha-1} dy\right)  du\\
        &\leq C   v^{-\alpha}(s-v)^{\alpha+1},
    \end{align*}
    so we have $g_2(s,s)=0.$ Note that
      \begin{align*}
        &\ K_2\\
        &=C\int_0^1 s^\alpha dW_s \int_0^s \frac{r^{-\alpha}\sigma_s^{-1}\partial_x f_s(\psi_s,\phi_s)\int_0^s du\int_0^u dB^H_v \sigma_v}{(s-r)^{\alpha+1}}dr\\
        &\quad-C\int_0^1 s^\alpha dW_s \int_0^s \frac{r^{-\alpha}\sigma_r^{-1}\partial_x f_r(\psi_r,\phi_r)\int_0^s du\int_0^u dB^H_v \sigma_v }{(s-r)^{\alpha+1}}dr\\
        &= C\int_0^1 s^\alpha  \left( \int_0^sdu \int_0^u \sigma_v dB^H_v\right) \left(\int_0^s  \frac{r^{-\alpha}\sigma_s^{-1}\partial_x f_s(\psi_s,\phi_s)-r^{-\alpha}\sigma_r^{-1}\partial_x f_r(\psi_r,\phi_r)}{(s-r)^{\alpha+1}} dr\right)dW_s   \\
        &=C \int_0^1 s^\alpha \left(\int_0^sdu \int_0^u v^{-\alpha}dW_v\int_v^u (x-v)^{\alpha-1}x^\alpha \sigma_x dx    \right)\\
        &\qquad\qquad\left(\int_0^s  \frac{r^{-\alpha}\sigma_s^{-1}\partial_x f_s(\psi_s,\phi_s)-r^{-\alpha}\sigma_r^{-1}\partial_x f_r(\psi_r,\phi_r)}{(s-r)^{\alpha+1}} dr\right)dW_s \\
        &=C\int_0^1 dW_s \int_0^s dW_v g_3(s,v),
    \end{align*}
where
\begin{align*}
    g_3(s,v)&=C\int_v^s      s^\alpha  v^{-\alpha}   \left( \int_v^u (x-v)^{\alpha-1}x^\alpha \sigma_x dx \right)\\
    &\qquad\qquad\cdot\left(\int_0^s  \frac{r^{-\alpha}\sigma_s^{-1}\partial_x f_s(\psi_s,\phi_s)-r^{-\alpha}\sigma_r^{-1}\partial_x f_r(\psi_r,\phi_r)}{(s-r)^{\alpha+1}} dr\right)du.
\end{align*}
   We have
\begin{align*}
    |g_3(s,v)| &\leq C \int_v^s s^\alpha v^{-\alpha} \left( \int_v^u (x-v)^{\alpha-1} x^\alpha \sigma_x \, dx \right) \\
    &\qquad\qquad \cdot \left( \int_0^s \frac{r^{-\alpha} \sigma_s^{-1} \partial_x f_s(\psi_s, \phi_s) - r^{-\alpha} \sigma_r^{-1} \partial_x f_r(\psi_r, \phi_r)}{(s - r)^{\alpha + 1}} \, dr \right) du \\
    &\leq C \int_v^s s^\alpha v^{-\alpha} u^\alpha (u - v)^\alpha \left( \int_v^u (x - v)^{\alpha - 1} x^\alpha \sigma_x \, dx \right) \\
    &\qquad\qquad \cdot \left( \int_0^s \frac{r^{-\alpha} (|s - r| + |\psi_s - \psi_r| + |\phi_s - \phi_r|)}{(s - r)^{\alpha + 1}} \, dr \right) du.
\end{align*}

We can follow the same approach to show that \( g_3(s, s) = 0 \). Observe that
\begin{align*}
    &K_3 \\
    &= C \int_0^1 s^\alpha \, dW_s\\
&\quad\cdot\int_0^s \frac{r^{-\alpha} \sigma_r^{-1} \partial_x f_r(\psi_r, \phi_r) \int_0^s du \int_0^u dB^H_v \sigma_v - r^{-\alpha} \sigma_r^{-1} \partial_x f_r(\psi_r, \phi_r) \int_0^r du \int_0^u dB^H_v \sigma_v}{(s - r)^{\alpha + 1}} \, dr \\
    &= C \int_0^1 s^\alpha \, dW_s \int_0^s \frac{r^{-\alpha} \sigma_r^{-1} \partial_x f_r(\psi_r, \phi_r) \int_r^s du \int_0^u dB^H_v \sigma_v}{(s - r)^{\alpha + 1}} \, dr \\
    &= C \int_0^1 dW_s \int_0^s dW_v \, g_4(s,v) + g_5(s,v).
\end{align*}

For \( |g_4(s,v)| \), we have:
\begin{align*}
    |g_4(s,v)| &= C \left| \int_v^s dr \int_r^s du \int_v^u s^\alpha \frac{r^{-\alpha} \sigma_r^{-1} \partial_x f_r(\psi_r, \phi_r) v^{-\alpha} (x - v)^{\alpha - 1} x^\alpha \sigma_x}{(s - r)^{\alpha + 1}} \, dx \right| \\
    &\leq C \left| \int_v^s dr \int_r^s s^\alpha \frac{r^{-\alpha} v^{-\alpha} (u - v)^{\alpha} u^\alpha}{(s - r)^{\alpha + 1}} \, du \right| \\
    &\leq C \left| v^{-\alpha} s^{2\alpha} \int_v^s \frac{r^{-\alpha} \left( (s - v)^{\alpha + 1} - (r - v)^{\alpha + 1} \right)}{(s - r)^{\alpha + 1}} \, dr \right| \\
    &\leq C \left| v^{-\alpha} (s - v)^{\alpha} s^{2\alpha} \int_v^s r^{-\alpha} (s - r)^{-\alpha} \, dr \right| \\
    &\leq C v^{-\alpha} (s - v)^\alpha s^{2\alpha}.
\end{align*}
Thus, \( g_4(s,s) = 0 \).

For \( |g_5(s,v)| \), we have:
\begin{align*}
    |g_5(s,v)| &= C \left| \int_0^v dr \int_r^s du \int_v^u s^\alpha \frac{r^{-\alpha} \sigma_r^{-1} \partial_x f_r(\psi_r, \phi_r) v^{-\alpha} (x - v)^{\alpha - 1} x^\alpha \sigma_x}{(s - r)^{\alpha + 1}} \, dx \right| \\
    &\leq C \left| \int_0^v dr \int_r^s du \, s^\alpha \frac{r^{-\alpha} v^{-\alpha} (x - v)^{\alpha - 1} x^\alpha}{(s - r)^{\alpha + 1}} \, dx \right| \\
    &\leq C \left| \int_0^v dr \int_r^s du \, s^\alpha \frac{r^{-\alpha} v^{-\alpha} (u - v)^\alpha u^\alpha}{(s - r)^{\alpha + 1}} \right| \\
    &\leq C \left| \int_0^v dr \int_r^s du \, s^\alpha \frac{r^{-\alpha} v^{-\alpha} (s - v)^\alpha (s^{\alpha + 1} - r^{\alpha + 1})}{(s - r)^{\alpha + 1}} \right| \\
    &\leq C \left| s^{1-\alpha} (s - v)^\alpha \int_0^1 y^{-\alpha} (1 - y)^{-\alpha} \, dy \right| \\
    &\leq C s^{1-\alpha} (s - v)^\alpha.
\end{align*}
Therefore, \( g_5(s,s) = 0 \).

According to Lemma~\ref{estimate2}, we have
\begin{equation*}
    \lim_{\varepsilon\to 0} \mathbb{E} \left[ \exp \left( \int_0^1 \partial_x f_s(\psi_s,\phi_s) \int_0^s dv \int_0^v \sigma_u dB^H_u \, dW_s \right) \,\middle|\, \left\| \int_0^\cdot \sigma_u \, dB^H_u \right\| \leq \varepsilon \right] = 0.
\end{equation*}
\end{proof}

    \subsection{Standard Case}
    We present the result of the standard case directly.
    \begin{theorem}  \label{result3}
Let $X$ be the solution of \eqref{fir}, and assume that the coefficients satisfy Assumption \((A)\).
For any $\phi=\psi' \in K_H^\sigma(L^2([0,1]))$ with $(\psi_0,\phi_0) =(x_0,y_0)$, when $  H =1/2$, the OM functional of $X$ with respect to the norms
$\Vert \cdot \Vert_{\beta+1}$ with $0<\beta<\frac{1}{2}-\frac{1}{2n}$ can be expressed as follows:
	\begin{equation*}
	J(\psi)=-\frac{1}{2} \int_0^1 \left(\frac{\phi_s' - f_s\left(\psi_s,\phi_s\right)}{\sigma_s}\right)^2+ \partial_yf_s(\psi_s,\phi_s)ds.
	\end{equation*}
    \end{theorem}
    \begin{proof}
       The proof in this case follows the same line as those for the previous two cases. Here, all fractional integrals and derivatives reduce to their standard counterparts, which substantially simplifies the argument.

    \end{proof}

   \begin{remark}
    For the standard  case, our result can be extended to the system
    \begin{equation}
        \begin{cases}
            dX_t = g_t(X_t,Y_t) \, dt, \\[4pt]
            dY_t = f_t(X_t,Y_t) \, dt + \sigma_t \, dW_t.
        \end{cases}
    \end{equation}
    This extension is feasible because the following estimate holds:
    \begin{align*}
        &\int_0^1 f_s(\psi_s,\phi_s)(X_s-\psi_s)\,dW_s \\
        =&\int_0^1 \frac{f_s(\psi_s,\phi_s)}{\sigma_s}(X_s-\psi_s)\,
           d\!\left(\int_0^s \sigma_u\,dW_u\right)\\
        =&\; \left(\int_0^1 \sigma_u\,dW_u\right)
           \frac{f_1(\psi_1,\phi_1)}{\sigma_1}(X_1-\psi_1) \\
          &-\int_0^1\!\left(\int_0^s \sigma_u\,dW_u\right)(X_s-\psi_s)\,
           \frac{d}{ds}\!\left(\frac{f_s(\psi_s,\phi_s)}{\sigma_s}\right)ds \\
          &-\int_0^1\!\left(\int_0^s \sigma_u\,dW_u\right)\!\bigl(
           g_s(X_s,Y_s)-g_s(\psi_s,\phi_s)\bigr)
           \frac{f_s(\psi_s,\phi_s)}{\sigma_s}\,ds \\
        &\to\;0 \qquad \text{as } \biggl\|\int_0^{\cdot} \sigma_u\,dW_u\biggr\|\to 0.
    \end{align*}
\end{remark}

\section*{Statements and Declarations}

\subsection*{Competing Interests}
The authors declare that they have no competing financial or non-financial interests that are directly or indirectly related to the work submitted for publication.

\section*{Data Availability Statement}
The numerical simulations in this study were generated for illustrative purposes. 
No supporting datasets are publicly available.

\bibliographystyle{plain}
\bibliography{math.bib}

\end{document}